\pdfoutput=1
\documentclass[letterpaper]{amsart}

\usepackage[utf8]{inputenc}
\usepackage{lmodern}
\usepackage[T1]{fontenc}
\usepackage{amssymb}
\usepackage{enumitem}
\usepackage{mathtools}
\usepackage{xr}

\usepackage{tikz-cd}
\usepackage[pdfusetitle]{hyperref}

\tikzcdset{arrow style=Latin Modern}
\usepackage{cleveref}

\mathtoolsset{mathic}

\newlist{enumarabic}{enumerate}{1}
\setlist[enumarabic]{font=\normalfont,label=(\arabic*),leftmargin=0.3in}
\newlist{enumroman}{enumerate}{1}
\setlist[enumroman]{font=\normalfont,label=(\roman*),leftmargin=0.3in}

\def\N{\mathbb N}
\def\Z{\mathbb Z}

\def\C{\mathbb C}

\def\ZZ{\mathcal Z}

\def\HH{H\mkern -2mu H}

\let\phi\varphi
\let\epsilon\varepsilon
\def\deg#1{|#1|}

\DeclareMathOperator{\length}{len}
\DeclareMathOperator{\sign}{sign}

\DeclareMathOperator{\Tor}{Tor}

\DeclareMathOperator{\Hom}{Hom}

\DeclareMathOperator*{\colim}{colim}

\let\shuffle\nabla
\def\AW{AW}
\def\AWhat{\skew5\widehat{\AW}}
\def\AWu#1{\AW_{\mkern -1mu #1}}
\def\AWhatu#1{\AWhat_{\mkern -2mu #1}}
\def\AWtilde{\skew4\widetilde{\AW}}
\def\AWtildeu#1{\AWtilde_{\mkern -3mu #1}}
\def\transposition#1#2{\tau_{#1}\cdot}
\def\cupi{\mathbin{\cup_{i}}}
\def\cupone{\mathbin{\cup_{1}}}
\def\cuptwo{\mathbin{\cup_{2}}}

\def\tilded{\tilde d}
\def\tildecup{\mathbin{\tilde\cup}}
\def\tildecupi{\mathbin{\tildecup_{i}}}
\def\tildecupone{\mathbin{\tildecup_{1}}}
\def\tildecuptwo{\mathbin{\tildecup_{2}}}

\def\kk{\Bbbk}
\def\Ll{\boldsymbol\Lambda}
\def\Sl{\mathbf S}
\def\Kl{\mathbf K}
\def\Su{\Sl^*}
\def\DJ{DJ}
\def\CP{\mathbb{CP}}
\DeclareMathOperator{\perm}{perm}
\DeclareMathOperator{\pos}{pos}

\def\ff{F}
\def\ffbar{f}

\def\susp{\mathrm{s}}
\def\desusp{\susp^{-1}}

\def\Simp{\mathcal S}
\def\SimpOne#1{\Simp'(#1)}
\def\SimpDJ{DJ}
\def\TopDJ{\mathcal{DJ}}

\def\transp#1{#1^{*}}
\let\transpp\transp

\let\newterm\emph

\newcommand*\xbar[1]{%
   \hbox{%
     \vbox{%
       \hrule height 0.35pt 
       \kern0.35ex
       \hbox{%
         \kern-0.15em
         \ensuremath{#1}%
         \kern-0.15em
       }%
     }%
   }%
}

\def\EE{\mathbf{E}}
\def\setone#1{\underline{#1}}
\def\setzero#1{[#1]}

\def\BB{\mathbf{B}}
\def\BBone{\mathbf{1}}

\def\refhomog#1{\ref*{h@#1}}
\makeatletter
\def\eqrefhomog#1{\textup{\tagform@{\ref*{h@#1}}}}
\makeatother
\def\citehomog#1{\cite[#1]{Franz:homog}}

\def\cf{\emph{cf.}}

\def\arxiv#1{\href{http://arxiv.org/abs/#1}{\texttt{arXiv:#1}}}
\def\doi#1{\href{http://doi.org/#1}{doi:#1}}

\let\osubsection\subsection
\def\subsection#1{\osubsection{\boldmath{#1}}}

\theoremstyle{plain}
\newtheorem{theorem}{Theorem}[section]
\newtheorem{proposition}[theorem]{Proposition}
\newtheorem{lemma}[theorem]{Lemma}
\newtheorem{corollary}[theorem]{Corollary}
\newtheorem{addendum}[theorem]{Addendum}

\theoremstyle{definition}
\newtheorem{remark}[theorem]{Remark}
\newtheorem{example}[theorem]{Example}

\theoremstyle{remark}
\newtheorem*{acknowledgements}{Acknowledgements}

\numberwithin{equation}{section}

\begin{document}

\title[Homotopy Gerstenhaber formality]%
  {Homotopy Gerstenhaber formality\\of Davis--Januszkiewicz spaces}
\author{Matthias Franz}
\thanks{The author was supported by an NSERC Discovery Grant.}
\address{Department of Mathematics, University of Western Ontario,
  London, Ont.\ N6A\;5B7, Canada}
\email{mfranz@uwo.ca}


\begin{abstract}
  A homotopy Gerstenhaber structure on a differential graded algebra
  is essentially a family of operations defining a multiplication on its bar construction.
  We prove that the normalized singular cochain algebra of a Davis--Janusz\-kie\-wicz space
  is formal as a homotopy Gerstenhaber algebra, for any coefficient ring.
  This generalizes a recent result by the author about classifying spaces of tori
  and also strengthens the well-known dga formality result for Davis--Janusz\-kie\-wicz spaces
  due to the author and Notbohm--Ray.
  As an application, we determine the cohomology rings
  of free and based loop spaces of Davis--Janusz\-kie\-wicz spaces.  
\end{abstract}

\maketitle

\section{Introduction}

Consider the classifying space~\(BT\) of a torus~\(T=(S^{1})^{n}\).
We write \(C^{*}(BT)\) for the differential graded algebra (dga)
of normalized singular cochains on~\(BT\) with coefficients in some commutative ring~\(\kk\).
Given that its cohomology~\(H^{*}(BT)\) is a polynomial algebra,
it is not difficult to show that \(C^{*}(BT)\) is formal in the sense
that it is quasi-isomorphic to~\(H^{*}(BT)\),
considered as a dga with trivial differential.
In 1974, Gugenheim and May~\cite[Thm.~4.1]{GugenheimMay:1974}
significantly sharpened this observation:

\begin{theorem}[Gugenheim--May]
  There is a quasi-isomorphism of dgas~\(C^{*}(BT) \to H^{*}(BT)\)
  annihilating all \(\cupone\)-products in~\(C^{*}(BT)\).
\end{theorem}

Up to sign, the \(\cupone\)-product is nothing but the
first operation of the homotopy Gerstenhaber structure on~\(C^{*}(BT)\)
constructed by Baues~\cite{Baues:1981}.
Recall that a homotopy Gerstenhaber algebra (hga)
is essentially a dga
equipped with operations
\begin{equation}
  E_{k}\colon A\otimes A^{\otimes k} \to A,
  \quad
  (a,b_{1},\dots,b_{k})\mapsto E_{k}(a;b_{1},\dots,b_{k})
\end{equation}
investing \(\BB A\), the reduced bar construction of~\(A\),
with an associative multiplication compatible with the diagonal.
Besides the Hochschild cochains of an associative algebra,
the cochain algebra~\(C^{*}(X)\) of a space~\(X\)
is the main example of an hga.
Any graded commutative dga is canonically an hga with trivial operations~\(E_{k}\)
for~\(k\ge1\).
We say that an hga~\(A\) is formal if it is quasi-isomorphic to its cohomology as an hga.

Recently we have
generalized the Gugenheim--May result as follows \citehomog{Thm.~\refhomog{thm:intro:hga-formality}}.

\begin{theorem}
  \label{thm:quiso-BT}
  There is a quasi-isomorphism of hgas \(C^{*}(BT)\to H^{*}(BT)\).
  In particular, \(C^{*}(BT)\) is formal as an hga.
\end{theorem}

In \Cref{sec:formality-BT} of the present paper we give a different proof of \Cref{thm:quiso-BT}.
It is based on a result which is of independent interest,
namely that certain interval cut operations commute with the shuffle map, see \Cref{thm:AW-shuffle}.

\smallskip

Let \(\Sigma\) be a simplicial complex on the finite vertex set~\(V\).
For each simplex~\(\sigma\in\Sigma\) we define
the Cartesian product~\(\DJ_{\sigma}=(\CP^{\infty})^{\sigma}\)
with factors indexed by~\(\sigma\).
By choosing a basepoint in~\(\CP^{\infty}\), we can consider \(\DJ_{\sigma}\)
as a subspace of~\((\CP^{\infty})^{V}\),
which is the classifying space of the torus~\(T=(S^{1})^{V}\).
The Davis--Janusz\-kie\-wicz space
\begin{equation}
  \DJ_{\Sigma} = \bigcup_{\sigma\in\Sigma}\DJ_{\sigma}
\end{equation}
is of great importance in toric topology,
\cf~\cite[Sec.~4.3]{BuchstaberPanov:2015} or~\cite{NotbohmRay:2005}.
Davis--Janusz\-kie\-wicz spaces are homotopy-equivalent to Borel constructions of
moment-angle complexes and of smooth toric varieties.

The cohomology algebra of~\(\DJ_{\Sigma}\)
is isomorphic to the evenly graded Stanley--Reisner ring~\(\kk[\Sigma]\).
It is known that Davis--Janusz\-kie\-wicz spaces are formal over any~\(\kk\).
This was established in the author's doctoral dissertation~\cite[Thm.~3.3.2]{Franz:2001},
see also~\cite[Thm.~1.1]{Franz:2003a},~\cite[Thm.~1.4]{Franz:2006},
and independently by Notbohm--Ray~\cite[Thm.~4.8]{NotbohmRay:2005}.\footnote{%
  Contrary to the claim made in~\cite[p.~43]{NotbohmRay:2005}, Notbohm--Ray's result is not more general
  since every simplicial complex canonically defines an open toric subvariety of complex affine space.}
From \Cref{thm:quiso-BT} we deduce:

\begin{theorem}
  \label{thm:DJ-formal-intro}
  There is a quasi-isomorphism of hgas \(C^{*}(DJ_{\Sigma})\to H^{*}(DJ_{\Sigma})\).
  In particular, \(C^{*}(DJ_{\Sigma})\) is formal as an hga.
\end{theorem}

In fact, this holds more generally for Davis--Janusz\-kie\-wicz spaces defined
from simplicial posets, see \Cref{thm:formality-DJ}.

As an application,
we derive formulas for the cohomology algebras
of based and free loop spaces of Davis--Janusz\-kie\-wicz spaces.
In a companion paper~\cite{Franz:torprod} we use our formality result
to determine the cohomology rings of smooth toric varieties
and of partial quotients of moment-angle complexes.

\begin{acknowledgements}
  I thank Dietrich Notbohm for
  suggesting a simplification of the proof of \Cref{thm:formality-DJ-simp}.
\end{acknowledgements}

\section{Preliminaries}

\subsection{General}

Throughout this paper, \(\kk\) denotes a commutative ring with unit, unless stated otherwise.
All modules and complexes will be over~\(\kk\), as well as all tensor products.
We denote the identity map on any module~\(C\) by~\(1_{C}\).
We also write \(\setone{n}=\{1,\dots,n\}\)
and \(\setzero{n}=\{0,1,\dots,n\}\).

Below we review several notions from homological algebra, mainly to fix notation
and our sign conventions.
The determination of signs will be important for us, for instance
in the proof of \Cref{thm:AW-shuffle}.

\subsection{Graded modules}

We write \(\deg{c}\in\Z\) for the degree of an element~\(c\) of a graded module~\(C\).
If \(B\) and~\(C\) are \(\Z\)-graded, then so is their tensor product~\(B\otimes C\)
as well as the module~\(\Hom(B,C)\) of homogeneous maps.
Recall that a map~\(f\colon B\to C\) has degree~\(n\) if it raises degrees by~\(n\).
We denote by~\(C^{*}=\Hom(C,\kk)\) the dual module of~\(C\).

Signs are governed by the Koszul sign rule, saying that
whenever two elements~\(b\),~\(c\) are transposed, then this produces
the sign~\((-1)^{\deg{b}\deg{c}}\).
For instance, the transposition map is given by
\begin{equation}
  \label{eq:def-transposition}
  T_{B,C}\colon B\otimes C\to C\otimes B,
  \quad
  b\otimes c\mapsto (-1)^{\deg{b}\deg{c}}\, c\otimes b.
\end{equation}

For any graded modules~\(A\),~\(B\),~\(C\),~\(D\), there is the canonical map
\begin{align}
  \label{eq:hom-tensor}
  \begin{split}
    \Hom(A,C)\otimes\Hom(B,D) &\to \Hom(A\otimes B,C\otimes D), \\
    f\otimes g &\mapsto \Bigl( a\otimes b \mapsto (-1)^{\deg{g}\deg{b}} f(a)\otimes g(b) \Bigr),
  \end{split}
\end{align}
which we will suppress from our notation.
It is compatible with composition of maps and satisfies
\begin{equation}
  \label{eq:transposition-maps}
  T_{C,D}\,(f\otimes g) = (-1)^{\deg{f}\deg{g}}\,(g\otimes f)\,T_{A,B}.
\end{equation}

For graded modules~\(B\) and~\(C\), the transpose is the map
\begin{equation}
  \label{eq:def-transpose}
    \Hom(B,C)\to\Hom(C^{*},B^{*}),
  \quad
  f\mapsto \Bigl(\,
  \transpp{f}\colon \gamma \mapsto (-1)^{\deg{f}\deg{\gamma}}\,\gamma\circ f
  \,\Bigr).
\end{equation}
For any~\(f\in\Hom(B,C)\) and~\(g\in \Hom(A,B)\) one has
\begin{equation}
  \transp{(f\circ g)} = (-1)^{\deg{f}\deg{g}}\, \transp{g}\circ\transpp{f},
\end{equation}
in particular
\begin{equation}
  \label{eq:transp-inv}
  \transpp{(f^{-1})} = (-1)^{\deg{f}}\,(\,\transpp{f})^{-1}
\end{equation}
if \(f\) is invertible, as well as
\begin{equation}
  \transp{(f\otimes g)} = \transpp{f}\otimes\transp{g}
  \in \Hom\bigl( (B\otimes D)^{*},(A\otimes C)^{*})
\end{equation}
for any~\(f\in\Hom(A,B)\) and~\(g\in\Hom(C,D)\). Moreover,
\begin{equation}
  \label{eq:transpose-transposition}
  \transp{T_{B,C}} = T_{C^{*},B^{*}}
\end{equation}
for all complexes~\(B\) and~\(C\).
Note that in the last two displayed equations we have
used the identification~\eqref{eq:hom-tensor}.

\subsection{Complexes}

We consider both homological and cohomological complexes, where the degree of the differential
is \(-1\) and~\(+1\), respectively.
A quasi-iso\-mor\-phism of complexes is a chain map inducing an isomorphism in homology,
and analogously for dgas and other dg~objects introduced below.

Let \(B\) and \(C\) be (both homological or cohomological) complexes.
The tensor product~\(B\otimes C\) is a complex
with differential
\begin{equation}
  d_{B\otimes C}=d_{B}\otimes 1+1\otimes d_{C},
\end{equation}
and \(\Hom(B,C)\) is a complex with differential
\begin{equation}
  d(f) = d\circ f - (-1)^{\deg{f}}f\circ d
\end{equation}
for any homogeneous map~\(f\colon B\to C\).
In particular, one has
\begin{equation}
  \label{eq:def-d-dual}
  d_{C^{*}}=-\transp{d_{C}},
\end{equation}
which means that
the differential of an element~\(\gamma\in C^{*}\) is given by
\begin{equation}
  (d\gamma)(c) = - (-1)^{\deg{\gamma}}\,\gamma(dc)
\end{equation}
for any~\(c\in C\).
The transposition map~\eqref{eq:def-transposition},
the identification~\eqref{eq:hom-tensor} and the transpose~\eqref{eq:def-transpose}
are chain maps.

Let \(f\),~\(g\colon B\to C\) be chain maps.
An~\(h\in\Hom_{\pm1}(B,C)\) is a homotopy from~\(f\) to~\(g\) if \(d(h)=d\circ h+h\circ d=g-f\).

The suspension map~\(\susp\colon C\to\susp\,C\) increases the degree of each element by~\(1\).
Following our sign rule, we have
\begin{equation}
  d(\susp\, c) = -\susp(dc)
\end{equation}
for any~\(c\in C\). Note that by~\eqref{eq:transposition-maps} we also have
\begin{equation}
  \label{eq:susp-transposition}
  (\susp\otimes\susp)\,T_{C,C} = - T_{\susp\,C,\susp\,C}\,(\susp\otimes\susp).
\end{equation}
If we think of the dual~\(C^{*}\) of a homological complex as a cohomological complex
(so that the cohomology of a space lives again in positive degrees),
then the degrees in~\((\susp\,C)^{*}\) are again increased by~\(1\) compared to~\(C^{*}\).
We fix the isomorphism~\((\susp\,C)^{*}\cong\susp\,C^{*}\) that makes the diagram
\begin{equation*}
  \begin{tikzcd}
    (\susp\,C)^{*} \arrow{dr}[left,pos=0.6]{\transp{\susp}\;\;} \arrow{rr}{\cong} & & \susp\,C^{*} \arrow{dl}[pos=0.35]{\!\!\!\desusp} \\
    & C^{*}
  \end{tikzcd}
\end{equation*}
commute. Using this identification and~\eqref{eq:transp-inv},  we have
\begin{equation}
  \label{eq:transp-susp}
  \transp{\susp} = \desusp
  \qquad\text{and}\qquad
  \transp{(\desusp)} = - (\,\transp{\susp})^{-1} = -\susp.
\end{equation}

\subsection{Algebras and coalgebras}

A general reference for this material is \cite[Sec.~1]{Munkholm:1974}.

A \newterm{differential graded algebra} (dga) is a graded associative algebra~\(A\) with unit~\(1\in A\)
such that the multiplication map~\(\mu_{A}\colon A\otimes A\to A\) is a chain map.
We sometimes think of the unit as a chain map~\(\eta_{A}\colon\kk\to A\).
An augmentation of~\(A\) is a morphism of dgas~\(\epsilon_{A}\colon A\to\kk\).
A dga~\(A\) is called (graded) commutative if \(\mu_{A}=T_{A,A}\,\mu_{A}\).

The notion of a coalgebra is dual to that of an algebra.
A \newterm{differential graded coalgebra} (dgc) is coassociative coalgebra~\(C\)
with an augmentation~\(\epsilon_{C}\colon C\to\kk\)
such that the diagonal~\(\Delta_{C}\colon C\to C\otimes C\) and \(\epsilon\)
are chain maps. It is called (graded) cocommutative if \(\Delta_{C}=T_{C,C}\,\Delta_{C}\).
A coaugmentation is a dgc map~\(\eta_{C}\colon\kk\to C\).

The dual of a dgc~\(C\) is canonically a dga with structure maps
\begin{equation}
  \mu_{C^{*}}=\transpp{\Delta_{C}}\colon C^{*}\otimes C^{*}\to (C\otimes C)^{*}\to C^{*}
  \qquad\text{and}\qquad
  \eta_{C^{*}} = \transp{\epsilon_{C}}.
\end{equation}

We say that an \(\N\)-graded dga~\(A\) is connected if \(A^{0}=\kk\),
and simply connected if additionally \(A^{1}=0\). The same applies to dgcs.

A \newterm{dg~bialgebra} is a complex~\(A\) that is both a dga and a dgc such
that the dga structure maps are coalgebra maps or, equivalently,
such that the dgc structure maps are algebra maps.
In this case, if \(C\) is a dg~\(A\)-module, then so is \(C\otimes C\)
via the diagonal~\(A\to A\otimes A\).
If \(C\) is additionally a dgc with \(A\)-equivariant structure maps,
we call it an \newterm{\(A\)-dgc}.
Here the \(A\)-action on~\(\kk\) is assumed to be via the augmentation.

Let \(A\) be a dga and \(C\) a dgc. The complex~\(\Hom(C,A)\) is a dga with
unit~\(\eta=\eta_{A}\,\epsilon_{C}\) and product given by
\begin{equation}
  f\cup g = \mu_{A}\,(f\otimes g)\,\Delta_{C}
\end{equation}
for~\(f\),~\(g\in\Hom(C,A)\).
An element~\(t\in\Hom_{\mp1}(C,A)\) such that
\begin{equation}
  d(t) = t\cup t,
  \qquad
  t\,\eta_{C}=0,
  \qquad
  \epsilon_{A}\,t = 0
\end{equation}
is called a \newterm{twisting cochain}.
A twisting cochain has degree~\(-1\) in the homological setting
and \(+1\) in the cohomological setting.

\subsection{Bar construction}

Let \(A\) be an augmented cohomological dga with augmentation ideal~\(\bar A\).
The (reduced) bar construction~\(\BB A\) of~\(A\) is the coaugmented dgc
\begin{equation}
  \BB A = \bigoplus_{k\ge 0} \BB_{k}A
  \qquad\text{with}\qquad
  \BB_{k}A = (\desusp\bar A)^{\otimes k},
\end{equation}
see for instance~\cite[\S 1.6]{Munkholm:1974} or~\cite[\S 0]{Baues:1981}.
The diagonal is given by the Alexander--Whitney formula
\begin{equation}
  \Delta [a_{1}|\cdots|a_{k}] = \sum_{i=0}^{k} [a_{1}|\cdots|a_{i}]\otimes[a_{i+1}|\cdots|a_{k}].
\end{equation}
In addition to the ordinary tensor product differential, the differential
on~\(\BB A\) has a component determined by the multiplication in~\(A\). More precisely,
\begin{align}
  \label{eq:bar-differential}
  d\,[a_{1}|\cdots|a_{k}] &= -\sum_{i=1}^{k}(-1)^{\epsilon_{i-1}}\,[a_{1}|\cdots|a_{i-1}|da_{i}|a_{i+1}|\cdots|a_{k}] \\
  \notag &\qquad\qquad +\sum_{i=1}^{\smash{k-1}} (-1)^{\epsilon_{i}}\,[a_{1}|\cdots|a_{i}a_{i+1}|\cdots|a_{k}]
\end{align}
where
\begin{equation}
  \epsilon_{i} = \deg{a_{1}} + \dots + \deg{a_{i}} - i.
\end{equation}
If \(A\) is \(1\)-connected, then \(\BB A\) is connected.

Let \(C\) be a connected coaugmented dgc and
\(f\colon C\to \BB A\) a morphism of coaugmented dgcs. Then the composition
\begin{equation}
  C \stackrel{f}{\longrightarrow} \BB A \longrightarrow \BB_{1}A = \desusp\bar A
  \stackrel{\susp}{\longrightarrow} \bar A \hookrightarrow A
\end{equation}
is a twisting cochain, and this assignment sets up a bijection between
morphisms of coaugmented dgcs~\(C\to \BB A\) and twisting cochains~\(C\to A\),
\cf~\cite[Prop.~1.9]{Munkholm:1974}.

\subsection{Simplicial sets}
\label{sec:simplicial}

Our reference for simplicial sets is \cite{May:1968}.

Let \(X\) be a simplicial set.
We call \(X\) \newterm{reduced} if \(X_{0}\) is a singleton
and \newterm{\(1\)-reduced} if \(X_{1}\) is a singleton.
The topological realization of~\(X\) is denoted by~\(|X|\).

We write \(C(X)\) and \(C^{*}(X)\) for the normalized
chain and cochain complex of~\(X\) with coefficients in~\(\kk\), \cf~\cite[Sec.~VIII.6]{MacLane:1967}.
Then \(C(X)\) is a dgc with the Alexander--Whitney map as diagonal
and augmentation induced by the unique map~\(X\to *\), and \(C^{*}(X)\) is a dga with product~\(\transpp{\Delta_{C(X)}}\).

For simplicial sets~\(X\) and~\(Y\), the shuffle map
\begin{equation}
  \shuffle=\shuffle^{X,Y}\colon C(X)\otimes C(Y)\to C(X\times Y)
\end{equation}
is a map of dgcs and furthermore satisfies \(\tau_{X,Y}\shuffle^{X,Y} = \shuffle^{Y,X} T_{C(X),C(Y)}\)
where \(\tau_{X,Y}\colon X\times Y\to Y\times X\) swaps the factors,
\cf~\cite[Sec.~3.2]{Franz:2003}.
We recall the shuffle map's definition for later reference.

Let \(x\) be a \(p\)-simplex in~\(X\) and \(y\) a \(q\)-simplex in~\(Y\).
We write a \((p,q)\)-shuffle as a pair of two surjective 
maps~\(\lambda\colon \setzero{p+q}\to\setzero{p}\) and
\(\mu\colon \setzero{p+q}\to\setzero{q}\)
such that \(\lambda(i-1)\le\lambda(i)\le\lambda(i-1)+1\) for~\(1\le i\le p+q\),
and analogously for~\(\mu\).
(We call such functions \newterm{slowly increasing}.)
Moreover, at each position~\(i\) either \(\lambda\) or \(\mu\) increases, but not both.
Then
\begin{equation}
  \label{eq:definition-shuffle}
  \shuffle^{X,Y}(x\otimes y) = \sum \sign(\lambda,\mu)\,\bigl(x(\lambda),y(\mu)\bigr),
\end{equation}
where the sum is over all \((p,q)\)-shuffles~\((\lambda,\mu)\).
Here \(x(\lambda)\) denotes the composition of~\(x\),
considered as a simplicial map from the standard \(p\)-simplex to~\(X\),
along~\(\lambda\), considered as a map from to the standard \((p+q)\)-simplex
to the standard \(p\)-simplex, and analogously for~\(y(\mu)\).

Let \(G\) be a simplicial group with multiplication~\(\mu\). Then \(C(G)\) is a dg~bialgebra
with unit given by the identity element of~\(G\) and multiplication
\begin{equation}
  C(G)\otimes C(G) \stackrel{\shuffle^{G,G}}{\longrightarrow} C(G\times G) \stackrel{\mu_{*}}{\longrightarrow} C(G).
\end{equation}
If \(G\) is commutative, then so is \(C(G)\).
Similarly, if \(G\) acts on the simplicial set~\(X\), then \(C(X)\) is a \(C(G)\)-dgc.
If the \(G\)-action is trivial, then the \(C(G)\)-action factors through the augmentation~\(\epsilon\colon C(G)\to\kk\).
(Remember that we use normalized chains.)
For~\(p\ge0\), we write \(1_{p}\in G_{p}\) for the identity element of the group of \(p\)-simplices.

\subsection{Universal bundles}
\label{sec:universal-bundles}

Given a simplicial group~\(G\),
we write \(\pi\colon EG\to BG\) for the universal \(G\)-bundle;
in~\cite[\S 21]{May:1968} the notation~\(BG=\xbar{W}(G)\) and~\(EG=W(G)\) is used.
We refer to~\citehomog{Sec~\refhomog{sec:universal-bundles}} for our conventions,
which can be obtained from~\cite{May:1968} by replacing \(G\) with the opposite group~\(G^{\mathrm{op}}\).
Important for us is the canonical map
\begin{equation}
  S\colon EG\to EG
\end{equation}
of degree~\(1\) such that for all~\(e\in EG_{p}\) one has
\begin{alignat}{2}
  \label{eq:d-S-0}
  \partial_{0} S e &= e, \\
  \partial_{1} S e &= e_{0} &\qquad& \text{if~\(p=0\),} \\
  \label{eq:d-S-k}
  \partial_{k} S e &= S \partial_{k-1} e &\qquad& \text{if \(p>0\) and \(k>0\).}
\end{alignat}
As a consequence, \(S\) induces a chain homotopy on~\(C(EG)\), again called \(S\),
from the projection to~\(e_{0}\) to the identity on~\(EG\),
\begin{equation}
  \label{eq:homotopy-S}
  (d S + S d)(e) =
  \begin{cases}
    e-e_{0} & \text{if \(p=0\),} \\
    e &  \text{if \(p>0\),}
  \end{cases}
\end{equation}
for any simplex~\(e\in EG\).
Furthermore, it satisfies
\begin{equation}
  \label{eq:properties-S}
  S S=0 \quad\text{and}\quad S e_0=0.
\end{equation}

\subsection{Topological spaces}

Let \(X\) be a topological space. Then \(\Simp(X)\), the graded set of singular simplices in~\(X\),
is a simplicial set. Topological groups are sent to simplicial groups.

Assume that \(X\) is (non-empty, connected and) simply connected with basepoint~\(*\in X\).
We write \(\SimpOne{X}\subset\Simp(X)\) for the simplicial subset of singular simplices
whose \(1\)-skeleton is mapped to~\(*\). The inclusion~\(\SimpOne{X}\hookrightarrow\Simp(X)\)
is a homotopy equivalence, compare~\cite[Thm.~12.5]{May:1968}.
In particular, the homotopy type of~\(\SimpOne{X}\) does not depend on the chosen basepoint.

\section{Interval cut operations}

We follow the treatment in~\cite[Secs.~1.2 \&~2.2]{BergerFresse:2004}.

Let \(k\),~\(l\ge0\). A surjection~\(u\colon\setone{k+l}\to\setone{l}\)
is called \newterm{non-degenerate} if \(u(i)\ne u(i+1)\) for all~\(1\le i<k+l\).
The surjections form an operad in the category of complexes.
The symmetric group~\(S_{l}\)
acts on the values of~\(u\) in the canonical way.
For the operadic composition rules and the differential
see~\cite[Secs.~1.2]{BergerFresse:2004}.
The \newterm{degree} of~\(u\) is~\(\deg{u}=k\).

Let \(X\) be a simplicial set.
The interval cut operation~\(\AWu{u}\colon C(X)\to C(X)^{\otimes l}\)
associated to a non-de\-gen\-er\-ate surjection~\(u\) is defined by
\begin{equation}
  \label{eq:def-interval-cut}
  \AWu{u}(x) = \AWu{u}^{X}(x) =
    \sum\pm x(\nu_{1})\otimes\cdots\otimes x(\nu_{l}),
\end{equation}
for~\(x\in X_{p}\), where the sum is over all cuts of the interval~\(\setzero{p}\),
say specified by indices
\begin{equation}
  0=p_0\le p_1\le\cdots\le p_{k+l}=p.
\end{equation}
The sequence~\(\nu_{s}\) is built from the concatenation of the intervals~\([p_{i-1},p_{i}]\)
such that \(u(i)=s\).
The sign is the product of a permutation sign and a position sign,
both of which are recalled below.

An interval~\([p_{i-1},p_{i}]\) is called \newterm{final}
if \(u\) does not assume the value~\(u(i)\) on arguments~\(j>i\).
Otherwise the interval is called \newterm{inner}.
By abuse of language, we sometimes refer to this by saying that \(u(i)\) is final or inner.
The \newterm{length} of the interval~\(I=[p_{i-1},p_{i}]\) is \(\length_{u}I=p_{i}-p_{i-1}\)
if the interval is final and \(p_{i}-p_{i-1}+1\) if it is inner.

The permutation sign is defined as follows:
Consider the shuffle that takes
the sequence~\((u(1),\dots,u(k+l))\)
to the ordered sequence~\((1,\dots,1,2,\dots,2,\dots,p,\dots,p)\).
The permutation of the associated intervals produces a sign according
to the Koszul sign rule with the lengths of the intervals
as degrees.
The position sign is the product of the signs~\((-1)^{p_{i}}\),
one for each inner interval~\([p_{i-1},p_{i}]\).

Note that \(x(\nu_{s})\) is degenerate if
two intervals corresponding to the same value of~\(u\),
say at positions~\(i<j\), overlap
in the sense that \(p_{i}=p_{j-1}\). In particular,
\(\AWu{u}\) vanishes on \(0\)-chains if \(\deg u>0\).

The assignment~\(u\mapsto\AWu{u}\) is a morphism from the surjection operad
to the endomorphism operad of~\(C(X)\). In other words, it is compatible
with the differentials, the actions of the symmetric groups
and the operadic composition law.

\begin{example}
  \label{ex:steenrod}
  We reproduce the examples from~\cite[\S 2.2.8]{BergerFresse:2004},
  giving explicit formulas for \(\cupone\)-products and~\(\cuptwo\)-products.

  Let \(x\in X\) be a \(p\)-simplex.
  The interval cut operation
  \begin{equation}
    \AWu{(1,2)}(x) = \sum_{0\le p_{1}\le p} x(0,\dots,p_{1})\otimes x(p_{1},\dots,p)
  \end{equation}
  is the Alexander--Whitney diagonal~\(\Delta\) of the dgc~\(C(X)\),
  hence its transpose defines the cup product in~\(C^{*}(X)\).
  We also have
  \begin{align}
    \AWu{(1,2,1)}(x) &= \, \sum_{\mathclap{p_{1}<p_{2}}} \,
      (-1)^{\epsilon_{1}} \, x(0,\dots,p_{1},p_{2},\dots,p)\otimes x(p_{1},\dots,p_{2}), \\
    \AWu{(1,2,1,2)}(x) &= \, \sum_{\mathclap{p_{1}<p_{2}<p_{3}}} \,
      (-1)^{\epsilon_{2}} \, x(0,\dots,p_{1},p_{2},\dots,p_{3})\otimes x(p_{1},\dots,p_{2},p_{3},\dots,p).
    \end{align}
    where the sign exponents are
    \begin{align}
      \epsilon_{1} &= (p-p_{2})(p_{2}-p_{1})+p_{1}, \\
      \epsilon_{2} &= (p_{3}-p_{2})(p_{2}-p_{1}+1)+p_{1}+p_{2}.
    \end{align}

  From the definition of the differential in the surjection operad \cite[\S 1.2.3]{BergerFresse:2004}
  we infer
  the identities
  \begin{align}
    d\bigl(\AWu{(1,2,1)}\bigr) &= \AWu{(2,1)}-\AWu{(1,2)} = T_{C(X),C(X)}\Delta-\Delta, \\
    d\bigl(\AWu{(1,2,1,2)}\bigr) &= \AWu{(2,1,2)}+\AWu{(1,2,1)} = T_{C(X),C(X)}\AWu{(1,2,1)}+\AWu{(1,2,1)}.
  \end{align}
  Thus, writing \(\beta\cupone\gamma=-\transpp{\AWu{(1,2,1)}}(\beta\otimes\gamma)\)
  and \(\beta\cuptwo\gamma=-\transpp{\AWu{(1,2,1,2)}}(\beta\otimes\gamma)\),
  we get
  \begin{align}
    \label{eq:cup-1}
    d(\beta\cupone\gamma)+d\beta\cupone\gamma+(-1)^{\deg{\beta}}\,\beta\cupone d\gamma
    &= \beta\cup\gamma - (-1)^{\deg{\beta}\deg{\gamma}}\,\gamma\cup\beta, \\
    \label{eq:cup-2}
    d(\beta\cuptwo\gamma)-d\beta\cuptwo\gamma-(-1)^{\deg{\beta}}\,\beta\cuptwo d\gamma
    &= \beta\cupone\gamma +(-1)^{\deg{\beta}\deg{\gamma}}\,\gamma\cupone\beta,
  \end{align}
  showing that \(\cupone=-\transpp{\AWu{(1,2,1)}}\) is a \(\cupone\)-product
  and \(\cuptwo=-\transpp{\AWu{(1,2,1,2)}}\) a \(\cuptwo\)-product in the usual sense.
\end{example}

\begin{remark}
  Up to sign, the \(\cupone\)-product and \(\cuptwo\)-product introduced above are exactly those
  used by Steenrod~\cite{Steenrod:1947}
  to construct Steenrod squares. To describe this relationship more precisely,
  we write the products and the differential on~\(C^{*}(X)\) used by Steenrod with a tilde.
  Taking the formulas~\eqref{eq:hom-tensor},~\eqref{eq:def-transpose} and~\eqref{eq:def-d-dual}
  as well as~\cite[Sec.~2]{Steenrod:1947} into account,
  one can see that Steenrod's definition is related to ours via
  \begin{equation}
    \tilded\,\beta = (-1)^{\deg{\beta}}\,d\,\beta
    \qquad\text{and}\qquad
    \beta\tildecupi\gamma = (-1)^{\deg{\beta}\deg{\gamma}+i}\,\beta\cupi\gamma.
  \end{equation}
  
  From this one can deduce that the map
  \begin{equation}
    \Theta\colon \bigl(C^{*}(X),d,\cup\bigr) \to \bigl(C^{*}(X),\tilded,\tildecup\bigr),
    \qquad
    \beta \mapsto (-1)^{\kappa(\deg{\beta})}\,\beta
  \end{equation}
  is a morphisms of dgas, where \(\kappa(p)\in\Z_{2}\) is defined by
  \begin{equation}
    \kappa(p) \equiv \frac{p(p+1)}{2}
    \equiv \begin{cases}
      0 & \text{if \(p\equiv 0,3 \pmod{4}\),} \\
      1 & \text{if \(p\equiv 1,2 \pmod{4}\).}
    \end{cases}
  \end{equation}
  Moreover, the identities
  \begin{align}
    \Theta(\beta\cupone\gamma) &= (-1)^{\deg{\beta}+\deg{\gamma}+1}\,\Theta(\beta)\tildecupone\Theta(\gamma), \\
    \Theta(\beta\cuptwo\gamma) &= - \Theta(\beta)\tildecuptwo\Theta(\gamma)
  \end{align}
  transform \eqref{eq:cup-1} and~\eqref{eq:cup-2} into the corresponding instances of~\cite[Thm.~5.1]{Steenrod:1947}.
  
  The \(\cupone\)-product additionally satisfies the Hirsch formula
  \begin{equation}
    \label{eq:hirsch-formula}
    (\alpha\cup\beta)\cupone \gamma = (-1)^{\deg{\alpha}}\,\alpha\cup(\beta\cupone \gamma) + (-1)^{\deg{\beta}\deg{\gamma}}\,(\alpha\cupone \gamma)\cup\beta,
  \end{equation}
  compare~\citehomog{eq.~\eqrefhomog{eq:hirsch-formula}}.
  Hirsch~\cite{Hirsch:1955} uses the same sign conventions as Steenrod.
  The isomorphism~\(\Theta\) maps \eqref{eq:hirsch-formula} to Hirsch's original identity~\cite[Thm.~1]{Hirsch:1955}.
\end{remark}

\smallskip

We mentioned that the shuffle map is a map of dgcs.
A similar, but weaker statement for a larger class
of interval cut operations than just \(\Delta=\AW_{(1,2)}\)
will be crucial for us. To state it, we need to introduce some more terminology.

We call a non-degenerate surjection~\(u\) \newterm{biased} if at most one value
is assumed more than once, and \newterm{strongly biased} if exactly one value
is repeated.
If this distinguished value equals \(1\), we call the surjection
\newterm{\(1\)-biased} or \newterm{strongly \(1\)-biased}.
For example, the surjection~\((1,2)\) 
is \(1\)-biased, and \((2,1,2)\) is strongly biased, but not strongly \(1\)-biased.
Note that for a biased surjection~\(u\), an interval~\([p_{i-1},p_{i}]\)
can only be inner if~\(u(i)\) is the distinguished value.

For a \(1\)-biased surjection~\(u\colon\setone{k+l+1}\to\setone{l+1}\) and any simplicial sets~\(X\)~and~\(Y\),
we define the map
\begin{equation}
  \label{eq:def-AW-tilde-u}
  \AWtildeu{u}^{X,Y} = \bigl((\pi_{X})_{*}\otimes(\pi_{Y})_{*}^{\otimes l}\bigr)\AWu{u}\colon C(X\times Y)\to C(X)\otimes C(Y)^{\otimes l},
\end{equation}
where \(\pi_{X}\)~and~\(\pi_{Y}\) are the canonical projections.

\begin{theorem}
  \label{thm:AW-shuffle}
  The following diagram commutes for all simplicial sets~\(X\), \(Y\),~\(Z\)
  and all \(1\)-biased surjections~\(u\colon\setone{k+l+1}\to\setone{l+1}\).
  \begin{equation}
  \begin{tikzcd}[column sep=huge,row sep=large]
      C(X)\otimes C(Y\times Z) \arrow{d}[left]{1\otimes\AWtildeu{u}^{Y,Z}} \arrow{r}{\shuffle^{X,Y\times Z}} & C(X\times Y\times Z) \arrow{d}{\AWtildeu{u}^{X\times Y,Z}} \\
      C(X)\otimes C(Y)\otimes C(Z)^{\otimes l} \arrow{r}[below]{\shuffle^{X,Y}\otimes 1^{\otimes l}} & C(X\times Y)\otimes C(Z)^{\otimes l} 
  \end{tikzcd}
  \end{equation}
\end{theorem}

\begin{proof}
  Pick (non-degenerate) simplices \(x\in X_{p}\) and
  \((y,z)\in(Y\times Z)_{q}\) and consider a simplex
  \begin{equation}
    \label{eq:term11}
    (x(\lambda),(y,z)(\mu)) = (x(\lambda),y(\mu),z(\mu)) \in (X\times Y\times Z)_{p+q}
  \end{equation}
  appearing in the sum~\eqref{eq:definition-shuffle}
  for~\(\shuffle^{X,Y\times Z}(x\otimes (y,z))\)
  where \(\lambda\colon\setzero{p+q}\to\setzero{p}\) and \(\mu\colon\setzero{p+q}\to\setzero{q}\).
  If we apply \(\AWtildeu{u}^{X\times Y,Z}\)
  to~\eqref{eq:term11}, then we get terms of the form
  \begin{equation}
    \label{eq:terms1}
    \bigl(x(\lambda\circ\nu_{1}),y(\mu\circ\nu_{1}) \bigr)
      \otimes z(\mu\circ\nu_{2})\otimes\cdots\otimes z(\mu\circ\nu_{l}).
  \end{equation}
  In order to obtain a non-degenerate simplex, we need that
  \(\mu\circ\nu_{2}\),~\ldots,~\(\mu\circ\nu_{l}\) are all strictly increasing.
  In other words, repeated indices can only occur in the
  subsequence~\(\mu\circ\nu_{1}\) of~\(\mu\), and \(\lambda\) therefore is constant
  on the images of~\(\nu_{2}\),~\dots,~\(\nu_{l}\).
  In what follows, we write \(\nu\) for the collection of
  sequences~\((\nu_{1},\dots,\nu_{l})\).

  Since \(u\) is \(1\)-biased, the sequence~\(\nu_{s}\) for~\(s\ge2\) enumerates the values
  in a single (final) interval~\([p_{i-1},p_{i}]\) of length~\(r_{s}=p_{i}-p_{i-1}\).
  In particular, \(\nu\) determines the interval cut.
  Write
  \begin{equation}
    \tilde q = q - \sum_{s\ge2}(r_{s}-1),
  \end{equation}
  so that \(\nu_{1}\colon \setzero{p+\tilde q}\to\setzero{p+q}\),
  and define
  \begin{align}
    \lambda'=\lambda\circ\nu_{1}&\colon \setzero{p+\tilde q}\to\setzero{p}, \\
    \nu'_{s}=\mu\circ\nu_{s}&\colon \setzero{r_{s}}\to\setzero{q}
    \qquad\text{for~\(s\ge2\)}.
  \end{align}
  Then \(\lambda'\) is slowly increasing and \(\nu_{s}\) strictly increasing for~\(s\ge2\).
  Moreover, we can write
  \begin{equation}
    \mu\circ\nu_{1}=\nu'_{1}\circ\mu'\colon \setzero{p+\tilde q}\to\setzero{q}
  \end{equation}
  for a unique strictly increasing sequence~\(\nu'_{1}\colon\setzero{\tilde q}\to\setzero{q}\)
  (defined on \(p\)~values fewer than~\(\nu_{1}\))
  and a unique slowly increasing sequence~\(\mu'\colon\setzero{p+\tilde q}\to\setzero{\tilde q}\).
  Then \((\lambda',\mu')\) is a \((p,\tilde q)\)-shuffle,
  and \eqref{eq:terms1} becomes
  \begin{equation}
    \label{eq:terms2}
    \bigl(x(\lambda'),y(\nu'_{1}\circ\mu') \bigr)
      \otimes z(\nu'_{2})\otimes\cdots\otimes z(\nu'_{l}),
  \end{equation}
  which, up to sign, is a term appearing in
  \begin{multline}
    \label{eq:functions2}
    (\shuffle^{X,Y}\otimes1^{\otimes l})(1\otimes\AWtildeu{u}^{Y,Z})\bigl(x\otimes (y,z)\bigr) \\
    = (-1)^{\deg{u}p}(\shuffle^{X,Y}\otimes1^{\otimes l})\bigl(x\otimes\AWtildeu{u}^{Y,Z}(y,z)\bigr).
  \end{multline}

  This procedure sets up a bijection between the terms
  of the form~\eqref{eq:terms1} and those of the form~\eqref{eq:terms2}
  because \((\lambda,\mu,\nu)\) can be reconstructed
  from~\((\lambda',\mu',\nu')\):
  We look at the indices~\(j\) on which \(\lambda'\) and~\(\mu'\) are defined,
  starting at~\(j=0\).
  When we reach the unique index~\(j\) such that \(\lambda'(j)=\mu'(j)=\nu_{2}(0)\),
  we insert \(r_{2}-1\)~values in the interval of definition of~\(\lambda'\) and~\(\mu'\)
  and modify these functions such that \(\lambda'\) is constant on the new indices and
  \(\mu'\) increases by~\(1\) with each new index. The values that \(\mu'\) previously
  assumed on later arguments are all increased by~\(r_{2}-1\), as is the range of~\(\mu'\).
  We repeat this procedure for all remaining values~\(s=3\),~\dots,~\(l\).
  At the end we obtain the sequences~\((\lambda,\mu,\nu)\)
  with which we have started.

  We now verify that for each term~\eqref{eq:terms1}\slash\eqref{eq:terms2}
  we get the same sign, which in each case is the product of
  a shuffle sign as well as a permutation and a position sign
  associated to the interval cut.
  We proceed by induction on~\(p\). For~\(p=0\) there is nothing to prove.

  To facilitate our discussion, we refer to the intervals in the interval cut
  we are considering
  as ``\(1\)-intervals'' and ``non-\(1\)-intervals'', depending on the corresponding
  value of~\(u\). All non-\(1\)-intervals are final. 

  For the given~\(p\),~\(q\) and~\((\lambda,\mu,\nu)\), we consider an index
  where \(\lambda\) actually increases, necessarily lying in some \(1\)-interval~\(I\)
  of the cut of the interval~\(\setzero{p+q}\).
  If we drop this index from~\(\lambda\) and also from~\(\nu_{1}\), we get a
  \((p-1,q)\)-shuffle~\((\hat\lambda,\hat\mu)\) and sequences~\(\hat\nu_{1}\),~\dots,~\(\nu_{l}\)
  for the corresponding interval cut of~\(\setzero{p+q-1}\). These sequences in turn
  define sequences~\((\hat\lambda',\hat\mu',\hat\nu')\).
  By induction, \((\hat\lambda,\hat\mu,\hat\nu)\) and \((\hat\lambda',\hat\mu',\hat\nu')\)
  lead to the same signs for any~\(\hat x\in X_{p-1}\) and~\(y\in Y_{q}\), \(z\in Z_{q}\).
  More precisely,
  \begin{equation}
    \frac
      {\sign(\hat\lambda,\hat\mu)\perm(\hat\nu)\pos(\hat\nu)}
      {\sign(\hat\lambda',\hat\mu')\perm(\hat\nu')\pos(\hat\nu')}
    = (-1)^{\deg{u}(p-1)}
  \end{equation}
  where \(\perm(-)\) and \(\pos(-)\) denote the permutation and position sign, respectively,
  of the corresponding interval cut.
  The sign on the right-hand side is the same as in~\eqref{eq:functions2}.
  We want to prove the analogous formula for~\((\lambda,\mu,\nu)\),
  with \(p-1\) replaced by~\(p\) in the exponent.

  The number of \(1\)-intervals is \(\deg{u}+1\); one of them is the interval~\(I\).
  Let \(m\) be the number of \(1\)-intervals strictly to the left of~\(I\),
  so that the number of \(1\)-intervals strictly to the right of~\(I\)
  equals \(\deg{u}-m\).
  Also, let \(I'_{1}\),~\dots,~\(I'_{n}\) be the non-\(1\)-intervals to the left of~\(I\).
  
  We have \(\hat\nu'=\nu'\), so that
  \begin{equation}
    \perm(\nu')\pos(\nu') = \perm(\hat\nu')\pos(\hat\nu').
  \end{equation}
  Let us compare the permutation and position signs for~\(\nu\) and~\(\hat\nu\).
  All positional signs of inner intervals
  from~\(I\) (including) to the end change by~\(-1\), that is,
  \begin{equation}
    \pos(\nu) = (-1)^{\deg{u}-m}\pos(\hat\nu).
  \end{equation}
  
  The permutation signs change for all pairs of intervals~\((I',I)\)
  where \(I'\) is a non-\(1\)-interval to the left of~\(I\).
  Each pair~\((I',I)\) produces a sign change by \(-1\) to the length of~\(I'\),
  hence
  \begin{equation}
    \perm(\nu) = \perm(\hat\nu)\prod_{i=1}^{n}(-1)^{\length_{u}I'_{i}}.
  \end{equation}
  
  The shuffle sign change,
  \begin{equation}
    \sigma =
    \frac{\sign(\lambda,\mu)}{\sign(\lambda',\mu')}
    :
    \frac{\sign(\hat\lambda,\hat\mu)}{\sign(\hat\lambda',\hat\mu')}
    \: ,
  \end{equation}
  is \(-1\) to the number of values in the non-\(1\)-intervals~\(I'_{1}\),~\dots,~\(I'_{n}\)
  that are cut out 
  when passing from~\((\lambda,\mu)\) to~\((\lambda',\mu')\).
  This is because \(\lambda\) is constant on these values,
  hence the permutation~\(\pi\) which orders the shuffle~\((\lambda,\mu)\)
  moves these indices past all indices where \(\lambda\) increases,
  and \(\sign(\lambda,\mu)=\sign\pi\).
  Given that the endpoints of each interval are shared with the neighbouring intervals,
  we have to count the number of values in the interior of each~\(I'_{i}\)
  plus the number of boundary values that are not the boundary of a \(1\)-interval.

  The number of interior values of the final interval~\(I'\) is \(\length_{u}I'_{i}-1\).

  There are \(m+n\)~intervals to the left of~\(I\), hence \(m+n+1\)~boundary values,
  including the left boundary of~\(I\).
  The \(1\)-intervals produce \(2m+1\)~of them,
  so that there are \(n-m\)~boundary values outside of \(1\)-intervals.
  In summary,
  \begin{equation}
    \sigma = (-1)^{m-n} \prod_{i=1}^{n}(-1)^{\length_{u}I'_{i}-1}
    = (-1)^{m}\prod_{i=1}^{n}(-1)^{\length_{u}I'_{i}} \: .
  \end{equation}

  Collecting all the sign changes, we get
  \begin{equation}
  \begin{split}
    \mathrlap{\frac
      {\sign(\lambda,\mu)\perm(\nu)\pos(\nu)}
      {\sign(\lambda',\mu')\perm(\nu')\pos(\nu')}} \qquad &
      \\
    &=
    \sigma\cdot(-1)^{\deg{u}-m}\cdot\prod_{i=1}^{n}(-1)^{\length_{u}I'_{i}}\cdot
    \frac
      {\sign(\hat\lambda,\hat\mu)\perm(\hat\nu)\pos(\hat\nu)}
      {\sign(\hat\lambda',\hat\mu')\perm(\hat\nu')\pos(\hat\nu')} \\
    &= (-1)^{\deg{u}+\deg{u}(p-1)}
    = (-1)^{\deg{u}p},
  \end{split}
  \end{equation}
  as desired.
  This completes the proof.
\end{proof}

\begin{corollary}
  \label{thm:AWtildeu-equiv}
  Let \(X\) and~\(Y\) be simplicial sets.
  If the simplicial group~\(G\) acts on~\(X\) (but not on~\(Y\)),
  then \(\AWtildeu{u}^{X,Y}\) is \(C(G)\)-equivariant
  for any \(1\)-biased surjection~\(u\).
\end{corollary}

\begin{proof}
  The actions of~\(C(G)\) on \(C(X)\) and \(C(X\times Y)\)
  are defined via the shuffle map.
  Now use \Cref{thm:AW-shuffle}.
\end{proof}

For a non-degenerate surjection~\(u\colon\setone{k+l+1}\to\setone{l+1}\) we write \(u'\) for
the non-degenerate surjection which is obtained by removing
the first value and, if it does not appear again,
decreasing all larger values by~\(1\). For example,
\((1,2,1)'=(2,1)\) and \((2,1,3)'=(1,2)\).
Also, for~\(1<s\le l\), we define the permutation
\begin{equation}
  \label{eq:def-tau-s}
  \tau_{s} =
  \begin{pmatrix}
    1 & 2 & 3 & \cdots & s \\
    s & 1 & 2 & \cdots & s-1
  \end{pmatrix}
  \in S_{l} \; ;
  \qquad
\end{equation}
for~\(s=1\) this is the identity permutation.

The following observation will be useful in proofs by induction.

\begin{lemma}
  \label{thm:AW-AW12}
  Let \(u\colon\setone{k+l}\to\setone{l}\) be a non-degenerate surjection
  such that \(s=u(1)\) is final. Then
  \begin{equation*}
    \AWu{u} = \transposition{s}{1}(1\otimes\AWu{u'})\,\Delta.
  \end{equation*}
\end{lemma}

\begin{proof}
  Given that the diagonal of~\(C(X)\) equals \(\AWu{(1,2)}\),
  our claim is equivalent to the identity
  \begin{equation}
    \AWu{u} = \transposition{s}{1} \bigl( \AWu{(1,2)}\circ_{2} \AWu{u'} \bigr).
  \end{equation}
  Since \(\AW\) is a morphism of operads,
  this follows from the laws for compositions
  and permutations of surjections.
  Alternatively, it can be checked directly.
\end{proof}

We also need to know how interval cut operations interact
with the canonical homotopy~\(S\) on the total space of a universal bundle
introduced in \Cref{sec:universal-bundles}.

\begin{lemma}
  \label{thm:AW-S}
  Let \(G\) be a simplicial group and
  let \(u\colon\setone{k+l+1}\to\setone{l+1}\) be non-de\-gen\-er\-ate with~\(u(1)=1\). Then on~\(C(EG)\)
  one has the identity
  \begin{equation*}
    \AWu{u}\,S = \begin{cases}
      (-1)^{\deg{u}}\,(S\otimes 1^{\otimes l})\,\AWu{u} + e_{0}\otimes\AWu{u'}\,S &
        \hbox{if \(u(1)\) is final,} \\
      (-1)^{\deg{u}}\,(S\otimes 1^{\otimes l})\,\AWu{u} + (S\otimes 1^{\otimes l})\,\AWu{u'}\,S &
        \hbox{if \(u(1)\) is inner.}
      \end{cases}
  \end{equation*}
  In particular,
  \begin{equation*}
    \Delta S = (S\otimes 1)\,\Delta + e_{0} \otimes S.
  \end{equation*}
\end{lemma}

\begin{proof}
  Let \(x\in EG_{p}\), and let \(\nu\colon\setzero{r}\to\setzero{p+1}\) be strictly increasing with~\(r>0\).
  By identifying \(\partial_{i}x\) with~\(x(\delta_{i})\) for suitable
  injections~\(\delta_{i}\colon\setzero{p-1}\to\setzero{p}\), \cf~\cite[p.~4]{May:1968},
  it follows from the relations~\eqref{eq:d-S-0}--\eqref{eq:d-S-k} that one has
  \begin{equation}
    (Sx)(\nu) = \begin{cases}
      x(\nu') & \text{if \(\nu(0)>0\),} \\
      S\bigl((Sx)(\hat\nu)\bigr) = S(x(\hat\nu')) & \text{if \(\nu(0)=0\)}
    \end{cases}
  \end{equation}
  where
  \begin{alignat}{2}
    \nu' &\colon\setzero{r}\to\setzero{p}, &\quad \nu'(i) &= \nu(i)-1, \\
    \hat\nu &\colon\setzero{r-1}\to\setzero{p+1}, &\quad \hat\nu(i) &= \nu(i+1), \\
    \hat\nu' &\colon\setzero{r-1}\to\setzero{p}, &\quad \hat\nu'(i) &= \nu(i+1)-1.
  \end{alignat}

  Now consider a term
  \begin{equation}
    \label{eq:SxJJ}
    \pm (Sx)(\nu_{1})\otimes(Sx)(\nu_{2})\otimes\dots\otimes(Sx)(\nu_{l+1})
  \end{equation}
  appearing in~\(\AWu{u}(Sx)\), coming from some cut of the interval~\(\setzero{p}\)
  determined by \(0=p_{0}\le p_{1}\le\dots\le p_{k+l+1}=p\).

  If \(p_{1}>0\), then we can write \eqref{eq:SxJJ} as
  \begin{equation}
    \pm S(x(\hat\nu_{1}'))\otimes x(\nu'_{2})\otimes\dots\otimes x(\nu'_{l+1}),
  \end{equation}
  which is, up to sign, a term appearing in~\((S\otimes1^{\otimes l})\,\AWu{u}(x)\).
  In fact, one gets all terms appearing in~\((S\otimes1^{\otimes l})\,\AWu{u}(x)\)
  this way. The permutation sign is the same as in~\(\AWu{u}(Sx)\), and the
  position sign changes by~\((-1)^{\deg{u}}\) because the right boundary of
  all intervals has decreased by~\(1\).

  Now assume \(p_{1}=0\) and that \(u(1)\) is final.
  Then \((Sx)(\nu_{1})=(Sx)(0)=e_{0}\), and we get exactly the terms
  in~\(e_{0}\otimes\AWu{u'}(Sx)\) with the same signs.

  Consider finally the case \(p_{1}=0\) and \(u(1)\) inner.
  This means that \(\nu_{1}(0)=0\) is the only value of~\(\nu_{1}\)
  coming from the first interval of the cut.
  Hence we can write \eqref{eq:SxJJ} as
  \begin{equation}
    \pm S\bigl((Sx)(\hat\nu_{1})\bigr)\otimes(Sx)(\nu_{2})\otimes\dots\otimes(Sx)(\nu_{l+1}).
  \end{equation}
  This is a term appearing in~\((S\otimes 1^{\otimes l})\,\AWu{u'}(Sx)\), and one gets all such terms
  with the same signs.

  The claim involving the diagonal of~\(C(EG)\)
  is the special case~\(u=(1,2)\).
\end{proof}

\section{Homotopy Gerstenhaber algebras}
\label{sec:hga}
\label{sec:hga-cochains}

A \newterm{homotopy Gerstenhaber algebra} (homotopy G-algebra, \newterm{hga})
is an augmented dga~\(A\) equipped with a certain multiplication
\( 
  \BB A \otimes \BB A \to \BB A
\) 
on the bar construction that together with the unit~\(\BBone\coloneqq[]\in \BB_{0}A\)
turns \(\BB A\) into a dg~bialgebra.
Specifically, we require that the components
\begin{equation}
  \EE_{kl}\colon \BB_{k}A \otimes \BB_{l}A \to \BB_{1}A = \desusp \bar A \stackrel{\susp}{\longrightarrow} \bar A \hookrightarrow A
\end{equation}
of the associated twisting cochain have the following form:
\(\EE_{10}\) and~\(\EE_{01}\) are the canonical map~\(\BB_{1}A\to A\),
and the remaining components vanish except possibly for~\(\EE_{1l}\)
with~\(l\ge1\). These conditions can be expressed in terms of the components
\begin{equation}
  E_{k} \colon A\otimes A^{\otimes k}
  \longrightarrow \bar A \otimes \bar A^{\otimes k}
  \xrightarrow{(\desusp)^{\otimes(k+1)}} \BB_{1}A \otimes \BB_{k}A \xrightarrow{\EE_{1k}} A
\end{equation}
(with the first map being the canonical projection along~\(\kk\)),
see~\cite[Sec.~3.2]{GerstenhaberVoronov:1995} or~\citehomog{Sec.~\refhomog{sec:def-hga}}.
We content ourselves with pointing out that the map~\(\cupone = - E_{1}\) satisfies the identities
\begin{gather}
  \label{eq:cupone-d}
  d(a\cupone b) + (da)\cupone b + (-1)^{\deg{a}}\, a\cupone (d b)
  = a\,b - (-1)^{\deg{a}\deg{b}}\,b\,a, \\
  \label{eq:cupone-hirsch}
  (ab)\cupone c
  = (-1)^{\deg{a}}\,a(b\cupone c)
  + (-1)^{\deg{b}\deg{c}}\,(a\cupone c)\,b
\end{gather}
for all~\(a\),~\(b\),~\(c\in A\),
which shows that it is a \(\cupone\)-product in the usual sense
satisfying the Hirsch formula.
This implies that the graded algebra~\(H^{*}(A)\) is commutative.

A morphism of hgas is a morphism of augmented dgas that is compatible
with the operations~\(E_{k}\) in the obvious way.

An hga~\(A\) with trivial operations~\(E_{k}=0\) for~\(k\ge1\)
is the same as a commutative augmented dga.
The product on~\(\BB A\) then is the usual shuffle product.
We say that an hga is \newterm{formal} if it is quasi-isomorphic
to~\(H^{*}(A)\) with the trivial hga structure.

The cohomology of any hga is a Gerstenhaber algebra with bracket
\begin{equation}
  \bigl\{[a],[b]\bigr\} = (-1)^{\deg{a}-1}\bigl[\,a\cupone b + (-1)^{\deg{a}\deg{b}}\,b\cupone a\,\bigr],
\end{equation}
\cf~\cite[Sec.~2.2]{GerstenhaberVoronov:1995}.
Obviously, a non-trivial Gerstenhaber bracket is an obstruction to formality.

Besides the Hochschild cochains of an associative algebra,
the cochain algebra of a simplicial set~\(X\) is the main example
of an hga, see~\cite[Sec.~2.3]{GerstenhaberVoronov:1995}.
In this case, the operations~\(E_{k}\) can be described by certain interval cut operations
\cite[\S 1.6.6]{BergerFresse:2004}
and essentially coincide for \(1\)-reduced~\(X\) with the multiplication on~\(\BB\,C^{*}(X)\)
constructed by Baues~\cite{Baues:1981}. Specifically, we have
\begin{equation}
  \label{eq:def-Ek}
  E_{k} = \transpp{\AWu{(1,2,1,3,1,\dots,1,k+1,1)}}
\end{equation}
for any~\(k\ge1\). In particular, the \(\cupone\)-product is given by~\(-\transpp{\AWu{(1,2,1)}}\),
which agrees with \Cref{ex:steenrod}. The existence of the \(\cuptwo\)-product
implies that the Gerstenhaber bracket in~\(H^{*}(X)\) is trivial.
Note that the interval cut operation appearing in~\eqref{eq:def-Ek} vanishes on \(0\)-simplices,
so that the image of its transpose lies in the augmentation ideal of~\(C^{*}(X)\)
for any choice of augmentation.

\begin{remark}
  \label{rem:other-hga-structure}
  In the definition of an hga
  there is an asymmetry between the first and the second factor
  for the product on~\(\BB A\).
  Some authors swap the factors, which entails changes in all formulas.
  The homotopy Gerstenhaber operations
  \begin{equation}
    \tilde E_{k}\colon C^{*}(X)^{\otimes k}\otimes C^{*}(X) \to C^{*}(X)
  \end{equation}
  obtained this way for a simplicial set~\(X\) satisfy
  \begin{equation}
    \label{eq:def-Ek-2}
  \begin{split}
    \tilde E_{k}
    &= \EE_{1k}\:T_{\BB_{k}C^{*}(X),\BB_{1}C^{*}(X)}\,(\desusp)^{\otimes(k+1)} \\
    &= (-1)^{k}\, \EE_{1k}\,(\desusp)^{\otimes(k+1)}\,T_{C^{*}(X)^{\otimes k},C^{*}(X)} \\
    &= (-1)^{k}\, \transpp{\AWu{(1,2,1,3,1,\dots,1,k+1,1)}}\,\transp{T_{C(X),C(X)^{\otimes k}}} \\
    &= (-1)^{k}\, \transp{\bigl(T_{C(X),C(X)^{\otimes k}}\,\AWu{(1,2,1,3,1,\dots,1,k+1,1)}\bigr)} \\
    &= (-1)^{k}\,\transpp{\AWu{(k+1,1,k+1,2,k+1,\dots,k+1,k,k+1)}},
  \end{split}
 \end{equation}
 where we have used \eqref{eq:susp-transposition} and~\eqref{eq:transpose-transposition}.
 (This corresponds to Baues' original definition.)
  For example, the \(\cupone\)-product is given by~\(\transpp{\AWu{(2,1,2)}}\) in this case.
  Our results in the following sections are valid for both kinds
  of homotopy Gerstenhaber structures.
\end{remark}

\section{Formality of~\texorpdfstring{\(BT\)}{BT}}
\label{sec:formality-BT}
\label{sec:construction-f}

Let \(T\) be a torus of rank~\(n\), considered as a simplicial group.
For example, \(T\) can be a compact torus~\(\cong(S^{1})^{n}\)
or an algebraic torus~\(\cong(\C^{\times})^{n}\). It can also be a simplicial torus,
that is, isomorphic to the bar construction~\(BN\) of some lattice~\(N\cong\Z^{n}\).
In any case, \(BT\) will denote the simplicial bar construction of~\(T\)
reviewed in \Cref{sec:universal-bundles}.

In this section we construct a map~\(\ffbar\colon H(BT)\to C(BT)\)
and show that its dual~\(\ffbar^*\colon C^*(BT)\to H^*(BT)\)
is a quasi-iso\-mor\-phism of hgas.
The definition of~\(f\) follows \citehomog{Sec.~\refhomog{sec:BT-formality-dga}}.
However, the proof that \(f^{*}\) annihilates all hga operation~\(E_{k}\)
is based on \Cref{thm:AWtildeu-equiv}.
Recall from Section~\ref{sec:universal-bundles}
that \(\pi\colon ET\to BT\) is the universal \(T\)-bundle,
\(e_0\) the canonical basepoint of~\(ET\)
and \(S\) the canonical contracting homotopy of~\(C(ET)\).

The following objects will play in role in our construction:
the commutative and cocommutative bialgebra~\(\Ll=H(T)\),
the cocommutative coalgebra~\(\Sl=H(BT)\) and its dual algebra~\(\Su=H^{*}(BT)\).
We choose a basis \(x_{1}\),~\ldots,~\(x_{n}\) of~\(\Ll_{1}\cong\kk^{n}\)
and transfer it
to a basis~\(y_{1}\),~\ldots,~\(y_{n}\) of~\(\Sl_{2}\)
via the transgression isomorphism~\(\Sl_{2}\cong\Ll_{1}\).
This gives a \(\kk\)-basis~\(y_{\alpha}\) for~\(\Sl\) with~\(\alpha\in\N^{r}\);
we also write \(y_{\alpha}=1\) for~\(\alpha=0\).
In terms of this basis, the diagonal of~\(\Sl\) has the form
\begin{equation}
  \label{eq:diagonal-S}
  \Delta y_{\alpha} = \sum_{\beta+\gamma=\alpha} y_{\beta}\otimes y_{\gamma}.
\end{equation}

For~\(1\le i\le n\), we choose a representative~\(c_{i}\in C_{1}(T)\)
of~\(x_{i}\in H_{1}(T)\). We assume that each~\(c_{i}\)
is a linear combination of loops at~\(1\in T\), meaning that
\(\partial_{0}\sigma=\partial_{1}\sigma=1_{0}\) for each~\(\sigma\) appearing in~\(c_{i}\).
For instance, we can fix a decomposition of a compact torus~\(T\) into circles
and let \(c_{i}\) can be the loop that starts at~\(1\)
and wraps with constant speed once around the \(i\)-th circle factor.
If \(T\) is a simplicial torus, then it is reduced,
so that any \(1\)-simplex is a loop.

Given that the dga~\(C(T)\) is commutative,
the~\(c_{i}\)'s define a quasi-iso\-mor\-phism of dg~bialgebras
\begin{equation}
  \label{eq:quiso-T}
  \phi\colon \Ll \to C(T),
  \quad
  x_{i}\mapsto c_{i}.
\end{equation}

Let \(\Kl\) be the homological Koszul complex associated to~\(\Ll\) and~\(\Sl\).
As a coalgebra, it is the tensor product
\begin{equation}
  \Kl = \Ll\otimes\Sl
\end{equation}
of the coalgebras~\(\Ll\) and~\(\Sl\),
and the \(\Ll\)-action is the canonical one on the first factor.
We write elements of~\(\Kl\) in the form~\(ay\) with~\(a\in\Ll\) and~\(y\in\Sl\).
We turn \(\Kl\) into a \(\Ll\)-dgc by defining
\begin{equation}
  d y_{i} = x_{i},\qquad
  d x_{i} = 0,
\end{equation}
which is independent of the chosen bases.
(The complex~\(\Kl\) is the twisted tensor product~\(\Ll\otimes_{t}\Sl\)
given by the twisting cochain~\(t\colon\Sl\to\Ll\)
corresponding to the isomorphism~\(\Sl_{2}\cong\Ll_{1}\),
\cf~\cite[Def.~30.8]{May:1968}.)
For an element~\(ay_{\alpha}\in\Kl\), we can write the differential as
\begin{equation}
  d(ay_{\alpha}) = \sum_i x_{i}a\,y_{\alpha|i}
\end{equation}
where the sum runs over the indices~\(i\) such that \(\alpha_{i}\ne0\),
and \(\alpha|i\) is obtained from~\(\alpha\) by decreasing \(\alpha_{i}\) by~\(1\).
Note that \(H(\Kl) = \kk\) and that
the map~\(\Kl\to\Sl\) induced by the augmentation~\(\Ll\to\kk\) is a map of dgcs.
We think of~\(\Kl\) as a model for~\(C(ET)\).

Via restriction along~\(\phi\), the dgc~\(C(ET)\) becomes a \(\Ll\)-dgc.
Following our previous work~\cite[Sec.~2.11]{Franz:2001},~\cite[Sec.~4.2]{Franz:2003},
we define a \(\Ll\)-equivariant map
\begin{equation}
  \ff\colon \Kl \to C(ET)
\end{equation}
recursively by setting
\begin{alignat}{2}
  \label{eq:def-f-0}
  \ff(1)       &= e_0,\\
  \label{eq:def-f-S}
  \ff(y) &= S\, \ff(d y) &\qquad& \text{if \(|y| > 0\)}, \\
  \label{eq:def-f-a}
  \ff(a y) &= a\cdot \ff(y) = \phi(a)\cdot \ff(y)
\end{alignat}
for~\(a\in\Ll\) and \(y\in\Sl\).

\begin{proposition}
  \label{thm:f-coalg}
  The map~\(\ff\) is a quasi-isomorphism of \(\Ll\)-dgcs.
\end{proposition}

\begin{proof}
  See~\cite[Prop.~4.3]{Franz:2003} or~\citehomog{Prop.~\refhomog{thm:f-coalg}}.
\end{proof}

For a \(1\)-biased surjection~\(u\colon\setone{k+l+1}\to\setone{l+1}\),
we now consider the map
\begin{equation}
  \AWhatu{u}
  = (1\otimes \pi_{*}^{\otimes l})\,\AWu{u}
  \colon C(ET)\to C(ET)\otimes C(BT)^{\otimes l}.
\end{equation}
Notice that \(\AWhatu{u}\) is the composition of the canonical map~\(C(ET)\to C(ET\times BT)\)
with the map~\(\AWtildeu{u}\) defined in~\eqref{eq:def-AW-tilde-u}.

\begin{lemma}
  \label{thm:AWhatu-f-zero}
  The composition
  \( 
    \AWhatu{u}\,\ff
    \colon \Kl\to C(ET)\otimes C(BT)^{\otimes l}
  \) 
  vanishes for all strongly \(1\)-biased surjections~\(u\colon\setone{k+l+1}\to\setone{l+1}\).
\end{lemma}

\begin{proof}
  We proceed by induction on~\(k+l\) and on the degree
  of~\(c=a y\in\Kl\), where \(a\in\Ll\) and \(y\in\Sl\).
  Since \(\AWu{u}\) vanishes on \(0\)-chains,
  there is nothing to prove for~\(|c|=0\),
  and also not for~\(k+l=0\) because there are no strongly biased surjections
  in this case.

  Assume \(\deg a>0\). By \Cref{thm:AWtildeu-equiv},
  the map~\(\AWhatu{u}\)
  is \(C(T)\)-equivariant. Hence
  \begin{equation}
    \AWhatu{u}\,\ff(a y)
    = \AWhatu{u}\bigl(\phi(a)\cdot \ff(y)\bigr)
    = \phi(a)\cdot \AWhatu{u}\,\ff(y)
     = 0
  \end{equation}
  by \eqref{eq:def-f-a} and induction.

  If \(u(1)\ne1\), then \(u(1)\) is final,
  hence \(u'\colon\setone{k+l}\to\setone{l}\) is strongly \(1\)-biased as well.
  \Cref{thm:AW-AW12} and \Cref{thm:f-coalg} together with induction
  therefore imply
  \begin{equation}
  \begin{split}
    \AWhatu{u} \ff(y)
    &= \transposition{u(1)}{1}\bigl(\pi_{*}\otimes \AWhatu{u'}\bigr)\Delta\,\ff(y) \\
    &= \transposition{u(1)}{1}\bigl(\pi_{*}\ff\otimes \AWhatu{u'}\,\ff\bigr)\,\Delta(y)
     = 0.
  \end{split}    
  \end{equation}

  It remains the case \(\deg{a}=0\) and \(u(1)=1\).
  Then \(u(1)\) is inner.
  Using definition~\eqref{eq:def-f-S} and \Cref{thm:AW-S}, we get
  \begin{equation}
  \begin{split}
    \AWhatu{u},\ff(y)
    &= (1\otimes \pi_{*}^{\otimes l})\,\AWu{u}\,S \ff(dy) \\
    &= (-1)^{\deg u}\,(S\otimes \pi_{*}^{\otimes l})\,\AWu{u}\,\ff(dy)
       + (S\otimes \pi_{*}^{\otimes l})\,\AWu{u'}\,S\,\ff(dy) \\
    &= (-1)^{\deg u}\,(S\otimes \pi_{*}^{\otimes l})\,\AWu{u}\,\ff(dy)
       + (S\otimes \pi_{*}^{\otimes l})\,\AWu{u'}\,\ff(y) \\
    &= (-1)^{\deg u}\,(S\otimes 1^{\otimes l})\,\AWhatu{u}\,\ff(dy) + (S\otimes 1^{\otimes l})\,\AWhatu{u'}\,\ff(y) \\
    &= (S\otimes 1^{\otimes l})\,\AWhatu{u'}\,\ff(y),
  \end{split}    
  \end{equation}
  again by induction.
  If \(u'\colon\setone{k+l}\to\setone{l}\) is strongly \(1\)-biased as well, then we are done by induction.

  If it is not, then \(\deg{u'}=0\), and \(\AWu{u'}\) is a permutation of
  the iterated diagonal \(C(ET)\to C(ET)^{\otimes l}\), \cf~\Cref{thm:AW-AW12}.
  \Cref{thm:f-coalg} therefore tells us that
  \(\AWu{u'}\,\ff(c)\) is a tensor product of terms~\(\ff(y')\) with~\(y'\in\Sl\),
  and the same holds true for~\(\AWhatu{u'}\).
  In particular, the first factor of the tensor product is of this form.
  The properties~\eqref{eq:properties-S} imply that \(S\) vanishes on such a term
  and hence so does \(\AWhatu{u}\) in this case.
\end{proof}

Since \(T\) acts trivially on~\(BT\),
the action of~\(\Ll\) on~\(C(BT)\) is given by the augmentation~\(\epsilon\colon\Ll\to\kk\).
The composition~\(\pi_{*}\ff\colon\Kl\to C(BT)\) therefore induces a dgc map
\begin{equation}
  \ffbar\colon \Sl = \kk\otimes_{\Ll}\Kl\to C(BT),
  \quad
  y\mapsto \pi_{*}\ff(y).
\end{equation}

\begin{theorem} \( \)
  \label{thm:formality-BT}
  \begin{enumroman}
  \item The dgc map~\(\ffbar\colon\Sl\to C(BT)\)
    induces the identity in homology, and
    \(\AWu{u}\ffbar\) vanishes
    for all strongly biased surjections~\(u\).
  \item The transpose~\(\transpp{\ffbar}\colon C^{*}(BT)\to\Su\)
    induces the identity in cohomology, and
    the composition~\(\transpp{\ffbar}\transpp{\AWu{u}}\)
    vanishes for all strongly biased surjections~\(u\).
    In particular, \(BT\) is homotopy Gerstenhaber formal.
  \item Fix a decomposition~\(T\cong(S^{1})^{n}\) into circles.
    For any~\(1\le i\le n\), choose a representative~\( c_{i}\) of~\(x_{i}\)
    that lies in the \(i\)-th circle factor.
    Then \(\ffbar\) and~\(\transpp{\ffbar}\) are  natural
    with respect to coordinatewise inclusions and projections of tori.
  \end{enumroman}
\end{theorem}

\begin{proof}
  That \(H(\ffbar)\) is the identity map is shown in
  the proof of~\citehomog{Thm.~\refhomog{thm:quiso-ffbar}}.

  Given a strongly biased surjection~\(u\colon\setone{k+l+1}\to\setone{l+1}\),
  choose a permutation \(\tau\in S_{l+1}\)
  such that \(\tilde u=\tau\cdot u\) is strongly \(1\)-biased.
  The composition~\(\AWu{\tilde u}\,\ffbar=(\pi_{*}\otimes 1^{\otimes l})\,\AWhatu{\tilde u}\,\ff\)
  vanishes by \Cref{thm:AWhatu-f-zero},
  hence so does \(\AWu{u}\,\ffbar\) by the equivariance of the
  interval cut operations with respect to the symmetric group.
  This proves our first claim.

  The second claim follows by dualizing,
  noting that the homotopy Gerstenhaber operations~\(E_{k}\)
  are the transposes of interval cut operations
  for certain strongly biased surjections, see formula~\eqref{eq:def-Ek}
  and \Cref{rem:other-hga-structure}.
  
  The naturality of~\(\ffbar\) and~\(\transpp{\ffbar}\)
  with respect to coordinatewise inclusions and projections
  follows directly from the recursive formulas~\eqref{eq:def-f-0}--\eqref{eq:def-f-a}.
  The homotopy operator~\(S\) commutes with Cartesian products,
  and the induced maps between the complexes~\(C(ET)\) are equivariant.
  (For projections, this uses again that we work with normalized chains.)
\end{proof}

\begin{addendum}
  \label{thm:cuptwo-BT}
  Assume that \(2\) is invertible in~\(\kk\). Then one can choose representatives~\(( c_{i})\)
  such that \(\transpp{\ffbar}\) furthermore annihilates all \(\cuptwo\)-products of cocycles.
\end{addendum}

\begin{proof}
  See~\citehomog{Prop.~\refhomog{thm:ffbar-cuptwo}}.
\end{proof}

\section{Davis\texorpdfstring{--}{-}Januszkiewicz spaces}
\label{sec:dj}

We now turn to Davis--Janusz\-kie\-wicz spaces.
The natural setting for our results is that of simplicial posets.
We recall the relevant definitions from~\cite[Secs.~2.8, 3.5, 4.10]{BuchstaberPanov:2015}.

A \newterm{simplicial poset} is a poset~\(\Sigma\) with initial element~\(\hat0\)
such that for any~\(\sigma\in\Sigma\) the interval~\([\hat0,\sigma]\)
is isomorphic to the face poset of a simplex.
The rank of~\(\sigma\) is the dimension of that simplex plus~\(1\).
By a \newterm{simplicial subposet} of~\(\Sigma\) we mean a subposet containing \([\hat0,\sigma]\)
for each~\(\sigma\) it contains.
Any simplicial complex is a simplicial poset, and its subcomplexes
are simplicial subposets.
We assume all our simplicial posets to be finite, with vertices contained
in a finite set~\(V\).
Note that we allow \newterm{ghost vertices}~\(v\in V\) not appearing in~\(\Sigma\).
By abuse of notation, we denote the complete simplicial complex
on the vertex set~\(V\) by the same letter.

The \newterm{folding map} of~\(\Sigma\) is the map~\(\Sigma\to V\)
that sends each~\(\sigma\in\Sigma\)
to the simplex determined by the vertices in~\(\sigma\).
A \newterm{vertex-pres\-erv\-ing morphism} of simplicial posets~\(\Sigma\to\Sigma'\)
on the same vertex set~\(V\)
is a map of posets that is compatible with the folding maps.
For example, the inclusion of a simplicial subposet is vertex-pre\-serv\-ing,
as is the folding map itself.

The \newterm{face ring}~\(\kk[\Sigma]\) of~\(\Sigma\) is the quotient
of the polynomial algebra~\(\kk[t_{\sigma}|\sigma\in\Sigma]\)
by the ideal generated by the elements
\begin{equation}
  t_{\hat0}-1,
  \qquad
  t_{\sigma}t_{\tau} - t_{\sigma\wedge\tau}\sum_{\!\!\rho\in\sigma\vee\tau\!\!}t_{\rho}
\end{equation}
with~\(\sigma\),~\(\tau\in\Sigma\).
Here \(\sigma\vee\tau\) denotes the join of~\(\sigma\) and~\(\tau\).
If it is empty, then the sum gives~\(0\), and \(t_{\sigma}t_{\tau}=0\) in~\(\kk[\Sigma]\).
If the join of~\(\sigma\) and~\(\tau\) is non-empty, then their meet is a unique element~\(\sigma\wedge\tau\in\Sigma\).
We grade \(\kk[\Sigma]\) by giving each~\(t_{\sigma}\) twice its rank as degree.
The face ring of a simplicial complex is its Stanley--Reisner ring
with the usual (even) grading.
A vertex-preserving map~\(\kappa\colon\Sigma\to\Sigma'\) induces the algebra map
\begin{equation}
  \kappa^{*}\colon \kk[\Sigma']\to\kk[\Sigma],
  \qquad
  t_{\sigma'} \mapsto \!\!\sum_{\kappa(\sigma)=\sigma'}\!\! t_{\sigma}.
\end{equation}
In particular, the folding map
makes \(\kk[\Sigma]\) an algebra over
the polynomial ring~\(\kk[V]\).

For any simplicial poset~\(\Sigma\) we have
\begin{equation}
  \label{eq:lim-SP}
  \kk[\Sigma] = \lim_{\sigma\in\Sigma}\kk[\sigma] \; ;
\end{equation}
moreover, for any simplicial subposets~\(\Sigma_{1}\),~\(\Sigma_{2}\subset\Sigma\)
such that~\(\Sigma=\Sigma_{1}\cup\Sigma_{2}\) one has a short exact sequence
\begin{equation}
  \label{eq:MV-SR}
  0 \longrightarrow \kk[\Sigma]
  \longrightarrow \kk[\Sigma_{1}]\oplus \kk[\Sigma_{2}]
  \longrightarrow \kk[\Sigma_{1}\cap\Sigma_{2}]
  \longrightarrow 0,
\end{equation}
\cf~\cite[Lemma~3.5.11]{BuchstaberPanov:2015} and its proof.
The first arrow above is the pair of the restriction maps,
and the second is the difference of the restriction maps.

Let \(S^{1}\) be a circle with classifying space~\(BS^{1}\),
and let \(T=(S^{1})^{V}\) so that \(BT=(BS^{1})^{V}\).
We choose a basepoint~\(*\in BS^{1}\). For a simplex~\(\sigma\in\Sigma\) we define
\begin{equation}
  \label{eq:def-DJ-sigma}
  \DJ_{\sigma} = (BS^{1})^{\sigma} \times \{*\}^{V\setminus\sigma},
\end{equation}
where the exponents indicate the indices in the Cartesian product, and
\begin{equation}
  \label{eq:def-DJ-SP}
  \DJ_{\Sigma} = \colim_{\sigma\in\Sigma}\DJ_{\sigma}.
\end{equation}
If \(\Sigma\) is a simplicial complex, then we have
\begin{equation}
  \DJ_{\Sigma} = \bigcup_{\sigma\in\Sigma}\DJ_{\sigma} \subset BT.
\end{equation}
We call \(\DJ_{\Sigma}\) the \newterm{Davis--Janusz\-kie\-wicz space} associated to~\(\Sigma\);
it is also denoted by~\((BS^{1},*)^{\Sigma}\) or~\(\ZZ_{\Sigma}(BS^{1},*)\).
It is simply connected and homotopy-equivalent to the Borel construction
of the moment-angle complex~\(\ZZ_{\Sigma}(D^{2},S^{1})\)
\cite[Thm.~4.3.2, Ex.~4.10.9]{BuchstaberPanov:2015}.
Since moment-angle complexes are equivariant strong deformation retracts
of complements of complex coordinate subspace arrangements
\cite[Prop.~20]{Strickland:1999}, \cite[Thm.~4.7.5]{BuchstaberPanov:2015},
this implies that Davis-Januszkiewicz spaces are also homotopy-equiv\-a\-lent
to Borel constructions of smooth toric varieties,
\cf~\cite[Thm.~5.4.5]{BuchstaberPanov:2015}.
We are going to reprove along the way that one has an isomorphism of graded algebras
\begin{equation}
  H^{*}(\DJ_{\Sigma}) \cong \kk[\Sigma],
\end{equation}
natural with respect to vertex-preserving morphisms.
In particular, the morphism~\(\DJ_{\Sigma}\to BT\) induces the map~\(\kk[V]\to\kk[\Sigma]\)
in cohomology.

Note that we can read \eqref{eq:def-DJ-sigma} and~\eqref{eq:def-DJ-SP}
both topologically and simplicially. In the simplicial context,
we choose \(S^{1}=B\Z\), so that \(\DJ_{\sigma}\) is \(1\)-reduced for any~\(\sigma\in\Sigma\)
and hence so is \(\DJ_{\Sigma}\).
Topologically, we choose \(BS^{1}=\CP^{\infty}\) with its usual CW~structure.
Then the Davis--Janusz\-kie\-wicz space becomes a CW~complex.
To distinguish the two constructions, we write the topological Davis--Janusz\-kie\-wicz space
as~\(\TopDJ_{\Sigma}\).

\begin{lemma}
  \label{thm:DJ-simp-top}
  There is a homotopy equivalence~\(\DJ_{\Sigma}\to\Simp(\TopDJ_{\Sigma})\),
  natural with respect to vertex-preserving morphisms.
\end{lemma}

\begin{proof}
  Since \(B(B\Z)\) and~\(\Simp(\CP^{\infty})\) are both Eilenberg--Mac\,Lane complexes,
  they are homotopy-equivalent, whence a homotopy equivalence
  \( 
  \DJ_{\sigma} \to \Simp(\TopDJ_{\sigma})
  \) 
  and therefore also a homotopy equivalence
  \begin{equation}
    |\DJ_{\sigma}| \to |\Simp(\TopDJ_{\sigma})| \to \TopDJ_{\sigma},
  \end{equation}
  natural in~\(\sigma\in\Sigma\).
  Given that the topological realization functor is left adjoint
  to the functor~\(X\mapsto \Simp(X)\) \cite[Thm.~16.1]{May:1968},
  we obtain homotopy equivalences
  \begin{equation}
    |\DJ_{\Sigma}| = \bigl|\colim_{\sigma\in\Sigma}\DJ_{\sigma}\bigr|
    = \colim_{\sigma\in\Sigma}|\DJ_{\sigma}|
    \to \colim_{\sigma\in\Sigma}\TopDJ_{\sigma}
    = \TopDJ_{\Sigma},
  \end{equation}
  and finally
  \begin{equation}
    \DJ_{\Sigma} \to \Simp(|\DJ_{\Sigma}|) \to \Simp(\TopDJ_{\Sigma}),
  \end{equation}
  natural with respect to vertex-preserving morphisms.
\end{proof}

\begin{theorem}
  \label{thm:formality-DJ-simp}
  Let \(\DJ_{\Sigma}\) be the simplicial Davis--Janusz\-kie\-wicz space
  associated to the simplicial poset~\(\Sigma\).
  There is a quasi-iso\-mor\-phism of hgas
  \begin{equation*}
    \ffbar^{*}_{\Sigma}\colon C^{*}(\DJ_{\Sigma})\to\kk[\Sigma],
  \end{equation*}
  natural with respect to vertex-preserving morphisms of simplicial posets.
  In particular, \(\DJ_{\Sigma}\) is homotopy Gerstenhaber formal.
  The map~\(\ffbar^{*}_{\Sigma}\) depends on the choice of a representative~\(c\in C_{1}(S^{1})\)
  of the canonical generator of~\(H_{1}(S^{1})=\kk\).
\end{theorem}

\begin{proof}
  \let\SimpDJ\DJ
  By embedding \(c\) into the various circle factors of~\(T=(S^{1})^{V}\),
  we obtain representatives~\( c_{v}\in C_{1}(T)\)
  of the canonical basis~\((x_{v})\) for~\(H_{1}(T)\).
  
  We note that limits of hgas exist by the same construction as for dgas.
  From \Cref{thm:formality-BT}
  we get a family of compatible hga quasi-iso\-mor\-phisms
  \begin{equation}
    \ffbar_{\sigma}^{*}\colon C^{*}(\SimpDJ_{\sigma})\to H^{*}(\SimpDJ_{\sigma})=\kk[\sigma],
  \end{equation}
  which assemble to an hga map
  \begin{equation}
    \ffbar_{\Sigma}^{*}\colon C^{*}(\SimpDJ_{\Sigma}) \to \lim_{\sigma\in\Sigma} C^{*}(\SimpDJ_{\sigma})\to \lim_{\sigma\in\Sigma} \kk[\sigma] = \kk[\Sigma],
  \end{equation}
  which is natural with respect to vertex-preserving morphisms.
  Note that the map on the left is an isomorphism
  because we work simplicially.

  To see that \(\ffbar^{*}_{\Sigma}\) is a quasi-iso\-mor\-phism
  also in the case where \(\Sigma\) does not have a single maximal simplex,
  we proceed by induction on the size of~\(\Sigma\).
  We write \(\Sigma=\Sigma_{1}\cup\Sigma_{2}\) as the union of two proper simplicial subposets
  and the following Mayer--Vietoris diagram
  \begin{equation}
  \label{eq:diagram-sp-1}
  \begin{tikzcd}[column sep=small]
    0 \arrow{r} & C^{*}(\SimpDJ_{\Sigma}) \arrow{d}{\ffbar^{*}_{\Sigma}} \arrow{r}
    & C^{*}(\SimpDJ_{\Sigma_{1}})\oplus C^{*}(\SimpDJ_{\Sigma_{2}}) \arrow{d}{(\ffbar^{*}_{\Sigma_{1}},\ffbar^{*}_{\Sigma_{2}})} \arrow{r}
    & C^{*}(\SimpDJ_{\Sigma_{1}\cap\Sigma_{2}}) \arrow{d}{\ffbar^{*}_{\Sigma_{1}\cap\Sigma_{2}}} \arrow{r}
    & 0 \\
    0 \arrow{r} & \kk[\Sigma] \arrow{r}
    & \kk[\Sigma_{1}]\oplus\kk[\Sigma_{2}] \arrow{r}
    & \kk[\Sigma_{1}\cap\Sigma_{2}] \arrow{r}
    & 0 \mathrlap{.}
  \end{tikzcd}
  \end{equation}
  The bottom row is the exact sequence~\eqref{eq:MV-SR}, and the top row is the analogous exact sequence
  for cochain complexes. Because \(\Sigma_{1}\),~\(\Sigma_{2}\) and \(\Sigma_{1}\cap\Sigma_{2}\)
  have fewer simplices than~\(\Sigma\), the middle and right vertical arrows
  are quasi-iso\-mor\-phisms by induction, hence so is \(\ffbar^{*}_{\Sigma}\)
  by the five-lemma.
\end{proof}

\begin{corollary}
  \label{thm:formality-DJ}
  \Cref{thm:formality-DJ-simp} remains true for the topological
  Davis--Janusz\-kie\-wicz space~\(\TopDJ_{\Sigma}\).
  The representative~\(c\) still refers to the simplicial circle.
\end{corollary}

\begin{proof}
  This follows from \Cref{thm:DJ-simp-top} and \Cref{thm:formality-DJ-simp}
  together with the naturality of the homotopy Gerstenhaber operations on cochains.
\end{proof}

\begin{addendum}
  \label{thm:cuptwo-DJ}
  Assume that \(2\) is invertible in~\(\kk\).
  The quasi-iso\-mor\-phism of hgas
  \( 
    f_{\Sigma}\colon C^{*}(\DJ_{\Sigma})\to\kk[\Sigma]
  \) 
  can furthermore be chosen
  to annihilate all \(\cuptwo\)-products of cocycles.
\end{addendum}

\begin{proof}
  This follows from \Cref{thm:cuptwo-BT}.
\end{proof}

We conclude with several immediate consequences
for loop spaces of Davis--Janusz\-kie\-wicz spaces.
A longer application appears in~\cite{Franz:torprod}
where we compute the cohomology rings of smooth toric varieties and
partial quotients of moment-angle complexes.
Additively, all the results below follow
from the dga formality of Davis--Janusz\-kie\-wicz spaces,
\cf~\cite[Sec.~8.4]{BuchstaberPanov:2015}.
We assume for the rest of this section that \(\kk\) is a principal ideal domain.

Baues~\cite[\S 2]{Baues:1981} has established a natural isomorphism of graded algebras
\begin{equation}
  \label{eq:cohomology-loops}
  H^{*}(\Omega\,|X|) = H^{*}(\BB \,C^{*}(X))
\end{equation}
for any \(1\)-reduced simplicial set~\(X\).
Here \(\Omega\,|X|\) denotes the Moore loop space of the topological realization of~\(X\).
If \(H^{*}(X)\) is free of finite type over~\(\kk\),
then \eqref{eq:cohomology-loops} is an isomorphism of bialgebras,
where the coproduct on the left-hand side is induced by the concatenation of loops
and the one on the right-hand side by the diagonal of the bar construction.

Recall that the loop space~\(\Omega X\) of a simply connected CW~complex~\(X\)
is homotopy equivalent to~\(\Omega\,|\SimpOne{X}|\).
(Here \(\SimpOne{X}\) denotes the \(1\)-reduced simplicial set of singular simplices
whose \(1\)-skeleton maps to a chosen basepoint of~\(X\).)
Note that the image of~\(\DJ_{\Sigma}\) in~\(S(\TopDJ_{\Sigma})\)
under the homotopy equivalence given by \Cref{thm:DJ-simp-top}
necessarily lies in~\(\SimpOne{\TopDJ_{\Sigma}}\)
with respect to the distinguished basepoint.

\let\DJ\TopDJ

\begin{proposition}
  \label{thm:iso-based-loops}
  There is an isomorphism of graded algebras
  \begin{equation*}
    H^{*}(\Omega\,\DJ_{\Sigma}) = \Tor_{\kk[\Sigma]}(\kk,\kk)
  \end{equation*}
  natural with respect to vertex-preserving morphisms of simplicial posets.
  If the homology of~\(\Omega\,\DJ_{\Sigma}\) is free over~\(\kk\),
  then this is an isomorphism of bialgebras.
\end{proposition}

By a result of Dobrinskaya \cite[Prop.~5.8]{Dobrinskaya:2009},
the freeness assumption holds if \(\Sigma\) is a shifted simplicial complex.

\begin{proof}
  Using~\eqref{eq:cohomology-loops} and \Cref{thm:formality-DJ},
  we have a chain of algebra isomorphisms
  \begin{equation}
    H^{*}(\Omega\,\DJ_{\Sigma}) = H^{*}\bigl(\BB \,C^{*}(\SimpOne{\DJ_{\Sigma}})\bigr)
    = H^{*}(\BB \, \kk[\Sigma]) = \Tor_{\kk[\Sigma]}(\kk,\kk)
  \end{equation}
  of algebras or bialgebras.
\end{proof}

We now turn to free loop spaces of Davis--Janusz\-kie\-wicz spaces.
Saneblidze~\cite[Sec.~6.1]{Saneblidze:2009} has described the cohomology ring
of the free loop space~\(L X\) of a simply connected space~\(X\) of finite type over~\(\kk\).
It is given by the Hochschild homology~\(\HH(C^{*}(\SimpOne X))\) together
with an explicit (non-associative) product on the Hochschild chains
in terms of the hga~structure of~\(C^{*}(\SimpOne X)\).\footnote{%
  Saneblidze claims the result for any commutative ring, without assuming finite type.
  However, the passage to cohomology after~\cite[Thm.~3]{Saneblidze:2009} appears to require some Künneth theorem.}
More generally, \(\HH(A)\) is naturally an algebra for any hga~\(A\). If \(A\) is graded commutative, then the product
again boils down to the shuffle product, \cf~\cite[Sec.~4.2]{Loday:1998} for the ungraded case.

\begin{proposition}
  There is an isomorphism of graded algebras
  \begin{equation*}
    H^{*}(L(\DJ_{\Sigma})) = \HH(\kk[\Sigma]),
  \end{equation*}
  natural with respect to vertex-preserving morphisms.
\end{proposition}

\begin{proof}
  This is analogous to the proof of \Cref{thm:iso-based-loops}. Note that like Baues, Saneblidze
  uses the `opposite' definition for the hga operations as described in \Cref{rem:other-hga-structure}.
  In his product formula~\cite[eq.~(6.5)]{Saneblidze:2009}, all terms except one
  are annihilated by the formality map
  \begin{equation}
    C^{*}(\SimpOne{\DJ_{\Sigma}}) \longrightarrow C^{*}(\SimpDJ_{\Sigma}) \stackrel{\ffbar^{*}_{\Sigma}}{\longrightarrow} \kk[\Sigma]
  \end{equation}
  since they contain hga operations.
  The remaining term, which corresponds to~\(p=0\) in the first sum of the formula, appears with the same sign as in the shuffle product.
  Hence the quasi-iso\-mor\-phism from the Hochschild chains for~\(C^{*}(\SimpOne{\DJ_{\Sigma}})\) to those for~\( \kk[\Sigma]\)
  is multiplicative.
\end{proof}

Let \(\ZZ_{\Sigma}\) be the moment-angle complex associated to a simplicial complex~\(\Sigma\),
and let \(K\subset T\) be a freely acting closed subgroup.
The space~\(X=\ZZ_{\Sigma}/K\) is often called a partial quotient
(at least if \(K\) is a subtorus). It comes with a canonical action of the torus~\(L=T/K\).
The topological analogue of the Jurkiewicz--Danilov theorem can be stated
by saying that the map~\(H_{L}^{*}(X)=H^{*}(\DJ_{\Sigma})\to H^{*}(X)\)
is surjective if \(X\) is a toric manifold.
The following loop space analogue applies to many more partial quotients.

\begin{proposition}
  \label{thm:loops-partial-quotient}
  Let \(\Sigma\) be a simplicial complex with vertex set equal to~\(V\).
  Let \(X=\ZZ_{\Sigma}/K\) be a partial quotient and set \(L=T/K\).
  Then, as graded algebras,
  \begin{equation*}
    H^{*}(\Omega X) = H^{*}(\Omega\,\DJ_{\Sigma}) \bigm/ H^{>0}(L)\cdot H^{*}(\Omega\,\DJ_{\Sigma}).
  \end{equation*}
\end{proposition}

This is again an isomorphism of bialgebras
if the homology of~\(\Omega X\) is free over~\(\kk\).
Note that the by \Cref{thm:iso-based-loops} and naturality,
the map~\(H^{*}(L)\to H^{*}(\Omega\DJ_{\Sigma})\) corresponds to the composition
\begin{equation}
  \Tor_{H^{*}(BL)}(\kk,\kk) \to \Tor_{H^{*}(BT)}(\kk,\kk) \to \Tor_{\kk[\Sigma]}(\kk,\kk).
\end{equation}

\begin{proof}
  We start by observing that we have
  \begin{equation}
    X_{L} = (\ZZ_{\Sigma})_{T} \simeq\DJ_{\Sigma}
  \end{equation}
  because \(K\) acts freely on~\(\ZZ_{\Sigma}\).

  Looping the fibration~\(X\hookrightarrow X_{L}\to BL\) gives the fibration
  \begin{equation}
    \Omega X \hookrightarrow \Omega X_{L} \to \Omega BL \simeq L.
  \end{equation}  
  Because the vertex set of~\(\Sigma\) equals \(V\), the map
  \begin{equation}
    H^{2}(BL;\Z)\to H_{L}^{2}(X;\Z) \cong H^{2}(\DJ_{\Sigma};\Z)\cong H^{2}(BT;\Z)
  \end{equation}
  is split injective. Hence the map
  \( 
    \pi_{1}(\Omega X_{L})\to \pi_{1}(\Omega BL) = \pi_{1}(L)
  \) 
  is surjective.
  This gives rise to a section~\(L\to\Omega X_{L}\) of topological groups,
  so that
  \begin{equation}
    \Omega X_{L}\simeq\Omega X\times L
  \end{equation}
  as spaces, \cf~\cite[p.~331]{BuchstaberPanov:2015}.
  Consequently, the map \(H^{*}(\Omega X_{L})\to H^{*}(\Omega X)\)
  is surjective, and the Leray--Hirsch theorem applies.
\end{proof}


\begin{thebibliography}{99}

\bibitem{Baues:1981}
H.-J.~Baues,
\newblock The double bar and cobar constructions,
\newblock \textit{Compositio Math.}~\textbf{43} (1981), 331--341;
\newblock available at \url{http://www.numdam.org/item?id=CM_1981__43_3_331_0}

\bibitem{BergerFresse:2004}
C.~Berger, B.~Fresse,
\newblock Combinatorial operad actions on cochains,
\newblock \textit{Math.\ Proc.\ Camb.\ Philos.\ Soc.}~\textbf{137} (2004), 135--174;
\newblock \doi{10.1017/S0305004103007138}

\bibitem{BuchstaberPanov:2015}
V.~M.~Buchstaber, T.~E.~Panov,
\newblock \textit{Toric topology},
\newblock Amer.\ Math.\ Soc., Providence, RI 2015;
\newblock \doi{10.1090/surv/204}

\bibitem{Dobrinskaya:2009}
N.~Dobrinskaya,
\newblock Loops on polyhedral products and diagonal arrangements,
\newblock \arxiv{0901.2871}

\bibitem{Franz:2001}
M.~Franz,
\newblock Koszul duality for tori,
\newblock doctoral dissertation, Univ.\ Konstanz 2001,
\newblock available at \url{http://math.sci.uwo.ca/~mfranz/papers/diss.pdf}

\bibitem{Franz:2003}
M.~Franz,
\newblock Koszul duality and equivariant cohomology for tori,
\newblock \textit{Int.\ Math.\ Res.\ Not.}~\textbf{42} (2003), 2255-2303;
\newblock \doi{10.1155/S1073792803206103}

\bibitem{Franz:2003a}
M.~Franz,
\newblock On the integral cohomology of smooth toric varieties,
\newblock \arxiv{math/0308253v1}

\bibitem{Franz:2006}
M.~Franz,
\newblock The integral cohomology of toric manifolds,
\newblock \textit{Proc.\ Steklov Inst.\ Math.}~\textbf{252} (2006), 53--62;
\newblock \doi{10.1134/S008154380601007X}

\bibitem{Franz:homog}
M.~Franz,
\newblock The cohomology rings of homogeneous spaces,
\newblock \arxiv{1907.04777v3}

\bibitem{Franz:torprod}
M.~Franz,
\newblock The cohomology rings of smooth toric varieties and quotients of moment-angle complexes,
\newblock \arxiv{1907.04791v4};
\newblock to appear in \textit{Geom.\ Topol.}

\bibitem{GerstenhaberVoronov:1995}
M.~Gerstenhaber, A.~A.~Voronov,
\newblock Homotopy G-algebras and moduli space operad,
\newblock \textit{Internat.\ Math.\ Res.\ Notices}~\textbf{1995} (1995), 141--153;
\newblock \doi{10.1155/S1073792895000110}

\bibitem{GugenheimMay:1974}
V.~K.~A.~M.~Gugenheim, J.~P.~May,
\newblock On the theory and applications of differential torsion products,
\newblock \textit{Mem.\ Am.\ Math.\ Soc.}~\textbf{142} (1974);
\newblock \doi{10.1090/memo/0142}

\bibitem{Hirsch:1955}
G.~Hirsch,
\newblock Quelques propriétés des produits de Steenrod,
\newblock \textit{C. R. Acad. Sci. Paris}~\textbf{241} (1955), 923--925

\bibitem{Loday:1998}
J.-L.~Loday,
\newblock \textit{Cyclic homology}, 2nd ed.,
\newblock Springer, Berlin 1998;
\newblock \doi{10.1007/978-3-662-11389-9}

\bibitem{MacLane:1967}
S.~Mac\,Lane,
\newblock \textit{Homology},
\newblock Springer, New York 1975;
\newblock \doi{10.1007/978-3-642-62029-4}

\bibitem{May:1968}
J.~P.~May,
\newblock \textit{Simplicial objects in algebraic topology},
\newblock Chicago Univ.\ Press, Chicago 1992

\bibitem{Munkholm:1974}
H.~J.~Munkholm,
\newblock The Eilenberg--Moore spectral sequence and strongly homotopy multiplicative maps,
\newblock \textit{J.~Pure Appl.\ Algebra}~\textbf{5} (1974), 1--50;
\newblock \doi{10.1016/0022-4049(74)90002-4}

\bibitem{NotbohmRay:2005}
D.~Notbohm, N.~Ray,
\newblock On Davis--Januszkiewicz homotopy types I; formality and rationalisation,
\newblock \textit{Alg.\ Geom.\ Topology}~\textbf{5} (2005), 31--51;
\newblock \doi{10.2140/agt.2005.5.31}
  
\bibitem{Saneblidze:2009}
S.~Saneblidze,
\newblock The bitwisted Cartesian model for the free loop fibration,
\newblock \textit{Topology Appl.}~\textbf{156} (2009), 897--910;
\newblock \doi{10.1016/j.topol.2008.11.002}

\bibitem{Steenrod:1947}
N.~E.~Steenrod,
\newblock Products of cocycles and extensions of mappings,
\newblock \textit{Ann.\ of Math.\ (2)}~\textbf{48} (1947), 290--320;
\newblock \doi{10.2307/1969172}

\bibitem{Strickland:1999}
N.~P.~Strickland,
\newblock Notes on toric spaces,
\newblock unpublished notes (1999)

\end{thebibliography}
\end{document}